\documentclass[a4paper,10pt]{amsart}

\usepackage[T1]{fontenc}
\usepackage{lmodern}
\usepackage[utf8]{inputenc}
\usepackage[english]{babel}
\usepackage{geometry}

\usepackage{booktabs}
\usepackage{color}
\usepackage{amsthm}
\usepackage{amsmath,amssymb}
\usepackage{graphicx}
\usepackage{listings}
\usepackage{emptypage}
\usepackage{newlfont}
\usepackage{verbatim}
\usepackage{amssymb}
\usepackage{array}
\usepackage{mathtools}
\usepackage{afterpage}
\usepackage{csquotes}
\usepackage{tabularx}
\usepackage{xcolor}
\usepackage{soul}
\usepackage{etoolbox}
\usepackage{hyperref}

\newcommand{\numberset}{\mathbb}

\newcommand{\R}{\numberset{R}}

\newcommand{\ud}{\mathrm{d}}

\theoremstyle{plain} 
\newtheorem{thm}{Theorem} 
\newtheorem{cor}[thm]{Corollary}
  
\newtheorem{prop}[thm]{Proposition}
\newtheorem*{theorem*}{Teorema}

\theoremstyle{definition} 
\newtheorem{defn}[thm]{Definition}

\newtheorem{rem}[thm]{Remark}

\usepackage{setspace}
\usepackage{fancyhdr}
\usepackage{pdfpages}
\usepackage{standalone}
\usepackage{appendix}

\numberwithin{equation}{section}
\numberwithin{thm}{section}

\title[Expansive solutions]{\textbf{Expansive solutions and the boundary at infinity for the homogeneous $N$-body problem}}
\date{}


\thanks{All authors are affiliated to INDAM-GNAMPA research group. 
D. B. was supported by MUR - M4C2 1.5 of PNRR with grant no. ECS00000036.}

\author{Diego Berti}
\address{Dipartimento di Matematica “Giuseppe Peano”\\Università degli Studi di Torino}
\email{diego.berti@unito.it}

\author{Davide Polimeni}
\address{Dipartimento di Matematica “Giuseppe Peano”\\Università degli Studi di Torino}
\email{davide.polimeni@unito.it}

\author{Susanna Terracini}
\address{Dipartimento di Matematica “Giuseppe Peano”\\Università degli Studi di Torino}
\email{susanna.terracini@unito.it}

\begin{document}

\begin{abstract}
We investigate expansive solutions of the $N$-body problem in $\mathbb{R}^d$ ($d \ge 2$) driven by homogeneous Newtonian potentials of degree $-\alpha$. We establish the existence of half-entire expansive motions with prescribed initial configuration and asymptotic direction for a wide range of homogeneity exponents $\alpha$. Our approach is variational and relies on the minimization of a suitably  renormalized Lagrangian action, allowing us to treat in a unified framework the hyperbolic, parabolic, and hyperbolic-parabolic regimes in the sense of Chazy’s classification.

Beyond existence, we derive refined asymptotic expansions for all classes of expansive solutions, identifying higher-order correction terms and improving previously known growth estimates, including the classical Newtonian case $\alpha=1$. In particular, for hyperbolic-parabolic solutions,  we provide a detailed description of the interplay between linear escape of cluster centers and internal parabolic dynamics, extending the cluster scattering picture to general homogeneous potentials.

Finally, we interpret these solutions within the geometric framework of the Jacobi-Maupertuis metric and the weak KAM theory. In this perspective, expansive motions correspond to geodesic rays and calibrating curves for the associated Hamilton-Jacobi equation, yielding a dynamical characterization of the boundary at infinity and a refined description of global viscosity solutions.
\end{abstract}

\maketitle


\section{Introduction and main results}

In this paper, we study \emph{expansive} solutions of the $N$-body problem in the Euclidean space $\mathbb{R}^{d}$, with $d\ge 2$, driven by a homogeneous Newtonian potential of degree $-\alpha$, where $\alpha>0$ will be specified and restricted according to the dynamical regime under consideration. The potential is
\[
U^\alpha(x)=\sum_{i<j}\frac{m_i m_j}{|r_i-r_j|^\alpha},
\]
where $m_1,\dots,m_N>0$ and $x=(r_1,\dots,r_N)\in\mathbb{R}^{dN}$ is a configuration with $r_i\neq r_j$ for all $i\neq j$.

If $r_1,\dots,r_N\in\mathbb{R}^d$ are the positions of $N$ point masses $m_1,\dots,m_N$, Newton's equations associated with $U^\alpha$ read
\begin{equation}\label{eq:newton}
m_i\ddot r_i
=
-\alpha\sum_{\substack{j\neq i}} m_i m_j \frac{r_i-r_j}{|r_i-r_j|^{2+\alpha}},
\qquad i=1,\dots,N,
\end{equation}
or, in compact form,
\begin{equation}\label{eq:NBP}
\mathcal{M}\ddot x=\nabla U^\alpha(x),
\end{equation}
where $\mathcal{M}=\mathrm{diag}(m_1 I_d,\dots,m_N I_d)$ is the diagonal mass matrix and $\nabla$ denotes the Euclidean gradient in $\mathbb{R}^{dN}$.

Since \eqref{eq:NBP} is translation invariant, we fix the origin of the inertial frame at the center of mass and work on the (linear) configuration space with zero barycenter,
\[
\mathcal{X}
=
\left\{
x=(r_1,\ldots,r_N)\in\mathbb{R}^{dN}:\ \sum_{i=1}^N m_i r_i=0
\right\}.
\]
It is convenient to introduce the collision-free set
\[
\Omega
=
\left\{
x=(r_1,\dots,r_N)\in\mathcal{X}:\ r_i\neq r_j\ \text{for all}\ i\neq j
\right\},
\]
and its complement $\Delta=\mathcal{X}\setminus\Omega$, the collision set.

We adopt the classical variational viewpoint for singular Lagrangians and consider the Lagrangian $L:T\Omega\to\mathbb{R}\cup\{+\infty\}$ and Hamiltonian $H:T^*\Omega\to\mathbb{R}\cup\{-\infty\}$,
\[
L(x,v)=\frac12\|v\|_{\mathcal{M}}^2+U^\alpha(x),
\qquad
H(x,p)=\frac12\|p\|_{\mathcal{M}^{-1}}^2-U^\alpha(x),
\]
where the norm is induced by the mass inner product
\[
\langle x,y\rangle_{\mathcal{M}}
=
\sum_{i=1}^N m_i\langle r_i,s_i\rangle,
\qquad
x=(r_1,\ldots,r_N),\ y=(s_1,\ldots,s_N)\in\mathcal{X},
\]
and $\langle\cdot,\cdot\rangle$ is the standard scalar product in $\mathbb{R}^d$.

\medskip

In the classical Newtonian case $\alpha=1$, Polimeni and Terracini \cite{PolimeniTerracini} proved in 2024 the existence of (half-entire) expansive solutions in $\mathbb{R}^d$, $d\ge 2$, with prescribed initial configuration and prescribed asymptotic direction. Their strategy introduces a unified global variational framework based on the minimization of a \emph{renormalized} Lagrangian action, and yields the existence of all three Chazy types of expansive motions -- hyperbolic, parabolic, and hyperbolic-parabolic -- in the sense of the classical classification of Chazy \cite{Chazy}. This approach generalizes earlier work by Maderna and Venturelli, who established existence results for parabolic \cite{MadernaVenturelli_GloballyMinimizingParabolic} and hyperbolic \cite{MadernaVenturelli_HyperbolicMotions} motions by different methods.

The aim of the present paper is twofold. First, we extend the existence theory of \cite{PolimeniTerracini} to the $N$-body problem with general homogeneous potentials of degree $-\alpha$. Second, we provide sharper asymptotic expansions for the resulting expansive motions, refining the leading-order regimes (and the lower-order correction terms) as $t\to+\infty$.

\medskip

Despite its deceptively simple formulation, a complete understanding of the dynamics of the general $N$-body problem remains out of reach. A fundamental obstruction is the non-integrability for $N\ge 3$, a phenomenon already anticipated by Poincar\'e \cite{Poincare_SurLeProblemes}. Another major difficulty is the possible occurrence of finite-time singularities: a motion develops a singularity at time $t^*$ if it cannot be extended beyond $t=t^*$. Such singularities are typically associated with collisions, i.e.\ configurations in which some mutual distances vanish. In contrast, we shall deal with expanding solutions.

\begin{defn}\label{def:expansive}
A motion $\gamma:[t_0,+\infty)\to\mathcal{X}$ is said to be \emph{expansive} if all mutual distances diverge, namely
\[
|r_i(t)-r_j(t)|\longrightarrow+\infty
\qquad\text{as }t\to+\infty,\ \text{for all }i<j.
\]
Equivalently, $\gamma$ is expansive if $U^\alpha(\gamma(t))\to 0$ as $t\to+\infty$.
\end{defn}

By energy conservation, expansive motions can occur only at nonnegative energy levels.

To quantify the asymptotic separation, we set
\[
r(t)=\min_{i<j}|r_i(t)-r_j(t)|,
\qquad
R(t)=\max_{i<j}|r_i(t)-r_j(t)|.
\]
For positive functions $f$ and $g$, we write $f\approx g$ if there exist constants $0<c\le C$ such that
$c\le f/g\le C$.

In the Newtonian case ($\alpha=1$), the following classical results from the 1960s--1970s provide a fundamental description of the final evolution of expansive motions.

\begin{thm}\label{PollardTheorem}
\begin{itemize}
\item[$(i)$] (Pollard, 1967 \cite{Pollard_BehaviorOfGravitationalSystems})
Let $\gamma$ be a motion defined for all $t>t_0$. If $r(t)$ is bounded away from $0$, then $R(t)=O(t)$ as $t\to+\infty$. Moreover, $R(t)/t\to+\infty$ if and only if $r(t)\to 0$.

\item[$(ii)$] (Marchal--Saari, 1976 \cite{MarchalSaari_FinalEvolution})
Let $\gamma$ be a motion defined for all $t>t_0$. Then either $R(t)/t\to+\infty$ and $r(t)\to 0$, or there exists a configuration $a\in\mathcal{X}$ such that
\[
\gamma(t)=at+O(t^{2/3})
\qquad\text{as }t\to+\infty.
\]
In particular, for \emph{superhyperbolic} motions -- i.e.\ motions such that $\limsup_{t\to+\infty}R(t)/t=+\infty$ -- the quotient $R(t)/t$ diverges.

\item[$(iii)$] (Marchal--Saari, 1976 \cite{MarchalSaari_FinalEvolution})
Assume $\gamma(t)=at+O(t^{2/3})$ for some $a\in\mathcal{X}$, and that $\gamma$ is expansive. Then for each pair $i<j$ such that $a_i=a_j$, one has $|r_i(t)-r_j(t)|\approx t^{2/3}$.
\end{itemize}
\end{thm}

As a consequence, expansive motions cannot be superhyperbolic. Hence, any expansive motion satisfies $\gamma(t)=at+O(t^{2/3})$ for some asymptotic velocity $a\in\mathcal{X}$. Chazy \cite{Chazy} provided a classification of expansive motions (with zero barycenter) according to the asymptotic growth of mutual distances:
\begin{itemize}
\item[$(H)$] \emph{Hyperbolic}: $a\in\Omega$ and $|r_i(t)-r_j(t)|\approx t$ for all $i<j$;
\item[$(P)$] \emph{Parabolic} (completely parabolic): $a=0$ and $|r_i(t)-r_j(t)|\approx t^{2/3}$ for all $i<j$;
\item[$(HP)$] \emph{Hyperbolic-parabolic} (partially hyperbolic): $a\in\Delta$, $a\neq 0$.
\end{itemize}

To refine this asymptotic classification, we recall the notion of \emph{limit shape}. Since we work in $\mathcal{X}$ (zero barycenter), similarity transformations reduce to scalings and orthogonal maps. We say that $\gamma$ has a limit shape if there exists a family of similarities $S(t)$ of $\mathbb{R}^d$ such that $S(t)\gamma(t)\to a\neq 0$ as $t\to+\infty$.

The possible limit shapes associated with expansive motions are summarized as follows:
\begin{itemize}
\item in the hyperbolic case, the limit shape is a collision-free configuration $a\in\Omega$, and coincides with the asymptotic velocity $a=\lim_{t\rightarrow+\infty}\gamma(t)/t$;
\item in the parabolic case, the limit shape is a central configuration $b$, i.e.\ a critical point of $U$ restricted to the inertia ellipsoid
$\mathcal{E}=\{x\in\mathcal{X}:\|x\|_{\mathcal M}^2=1\}$;
\item in the hyperbolic-parabolic case, the limit shape again coincides with the asymptotic velocity, which now lies in the collision set $a\in\Delta$.
\end{itemize}

Working with the Newtonian potential of degree $-1$, Polimeni and Terracini \cite{PolimeniTerracini} extended, refined and unified earlier results by Maderna and Venturelli \cite{MadernaVenturelli_GloballyMinimizingParabolic,MadernaVenturelli_HyperbolicMotions}, proving the existence of half-entire expansive solutions in all three Chazy regimes $(H)$, $(P)$ and $(HP)$ via a single renormalized action-minimization scheme.

Their first result concerns half-entire hyperbolic solutions.

\begin{thm}[Maderna and Venturelli, 2020 \cite{MadernaVenturelli_HyperbolicMotions}]\label{thm_hyperbolic}
Given $d\ge 2$, for the Newtonian $N$-body problem in $\mathbb{R}^d$, there exists a hyperbolic solution $\gamma(t)$ defined for $t\in[1,+\infty)$ of the form
\begin{equation}\label{e:MV}
\gamma(t)=at-\log(t)\,\nabla^{\mathcal{M}}U(a)+O(1),
\qquad\text{as }t\to+\infty,
\end{equation}
for any initial condition $x=\gamma(1)\in\mathcal{X}$ and any collision-free configuration $a\in\Omega$.
\end{thm}

For the parabolic case, the result of \cite{MadernaVenturelli_GloballyMinimizingParabolic} was sharpened in \cite{PolimeniTerracini}:

\begin{thm}[Maderna and Venturelli, 2009 \cite{MadernaVenturelli_GloballyMinimizingParabolic}; Polimeni and Terracini, 2024 \cite{PolimeniTerracini}]\label{thm_parabolic}
Given $d\ge 2$, for the Newtonian $N$-body problem in $\mathbb{R}^d$, there exists a parabolic solution $\gamma:[1,+\infty)\to\Omega$ of the form
\begin{equation}\label{eq:parabolic_MVPT}
\gamma(t)=\beta b_m t^{2/3}+o(t^{1/3^+}),
\qquad\text{as }t\to+\infty,
\end{equation}
for any initial configuration $x=\gamma(1)\in\mathcal{X}$, any minimal normalized central configuration $b_m$, and $\beta=\sqrt[3]{\frac92\,U(b_m)}$.
\end{thm}

The improvement lies in the sharper remainder estimate: here $o(t^{1/3^+})$ means $o(t^{1/3+\varepsilon})$ for every $\varepsilon>0$, as $t\to+\infty$.

Polimeni and Terracini's variational framework also produced general existence results for hyperbolic-parabolic motions, extending earlier work such as \cite{Burgos_PartiallyHyperbolic}. Given $a\in\Delta\setminus\{0\}$, they consider the $a$-cluster partition induced by
\begin{equation}\label{eq:equivalence_relation}
i\sim j\quad\Longleftrightarrow\quad a_i=a_j,
\end{equation}
and define the associated clustered potentials by summing the internal potentials of each cluster.

\begin{thm}[Polimeni and Terracini, 2024 \cite{PolimeniTerracini}]\label{thm_partially_hyperbolic}
Given $d\ge 2$, for the Newtonian $N$-body problem in $\mathbb{R}^d$, there exists a hyperbolic-parabolic motion $\gamma:[1,+\infty)\to\Omega$ of the form
\begin{equation}\label{eq:hyperbolic_parabolic_PT}
\gamma(t)=at+\beta b_m t^{2/3}+o(t^{1/3^+}),
\qquad\text{as }t\to+\infty,
\end{equation}
for any initial configuration $x=\gamma(1)\in\mathcal{X}$, any collision configuration $a\in\Delta$, any normalized minimal central configuration $b_m$ of the $a$-clustered potential, and any $h>0$.
\end{thm}

The asymptotic decomposition in \eqref{eq:hyperbolic_parabolic_PT} reflects a cluster scattering mechanism: cluster centers of mass separate linearly, while internal cluster scales grow at the parabolic rate $t^{2/3}$ and converge (after renormalization) to minimal central configurations.

\begin{cor}[Polimeni and Terracini, 2024 \cite{PolimeniTerracini}]
The motions $\gamma(t)$ in Theorems \ref{thm_hyperbolic}, \ref{thm_parabolic}, and \ref{thm_partially_hyperbolic} are continuous at $t=1$ and collision-free for $t>1$. Moreover, they are free-time minimizers of the action at their respective energy levels.
\end{cor}

Recently, further existence results for expansive solutions have appeared applying to $-\alpha$-homogeneous solutions with $\alpha\in(0,2)$. For instance, Yu \cite{Yu_Hyperbolic} proved in 2024 the existence of hyperbolic motions $\gamma:[t_0,+\infty)\to\Omega$ for prescribed initial configuration, positive energy $h$, and collisionless asymptotic configuration $a$, of the form
\begin{equation}\label{eq:Yu_hyperbolic}
\gamma(t)=\sqrt{2h}\,a\,t+o(t),
\qquad\text{as }t\to+\infty.
\end{equation}
Yu's approach is based on free-time minimizers at fixed energy and Marchal's principle, building on ideas originating in \cite{LiuYanZhou}.
\medskip

As said, we extend the existence theory of expansive motions to homogeneous Newtonian potentials of arbitrary degree $-\alpha$, thus extending and refining Yu's results in \cite{Yu_Hyperbolic}. On the other hand, we obtain sharper asymptotic descriptions of such motions. Both results are achieved within a unified variational framework based on the renormalized action principle of Polimeni and Terracini, adapted to the homogeneous setting.

\medskip

\subsection{Main results of this paper.}
We now describe our contributions. Throughout, we distinguish the Euclidean gradient $\nabla U^\alpha$ in $\mathbb{R}^{dN}$ from the gradient with respect to the mass inner product:
\[
\nabla^{\mathcal{M}}U^\alpha(x)=\mathcal{M}^{-1}\nabla U^\alpha(x).
\]
Notice that in \cite{MadernaVenturelli_HyperbolicMotions,PolimeniTerracini} the notation $\nabla U$ is often used for $\nabla^{\mathcal M}U$.

For $\alpha\neq 1$ and $a\in\Omega$, we introduce the vector
\begin{equation}\label{e:Gamma1}
\Gamma_1=\Gamma_1(\alpha,a):=-\frac{\mathcal{M}^{-1}\nabla U^\alpha(a)}{\alpha(1-\alpha)}.
\end{equation}

\medskip

\noindent
We state our main results in Theorems~\ref{thm:hyp_alpha_expansive_1}, \ref{thm:hyp_alpha_expansive_2}, \ref{thm:par_alpha} and \ref{thm:HP_alpha}.

\begin{thm}[Hyperbolic motions with $\alpha>1/2$]\label{thm:hyp_alpha_expansive_1}
Let $x\in\mathcal{X}$ be arbitrary and let $\alpha\in(1/2,\infty)$. For any $a\in\Omega$ there exists a hyperbolic solution $\gamma(t)$ of \eqref{eq:NBP} defined for $t>1$, with $\gamma(1)=x$. Moreover, for proper constant vectors $Q,Q',Q''\in\mathcal{X}$, as $t\to+\infty$:
\begin{enumerate}
\item if $\alpha>1$,
\[
\gamma(t)=at+\Gamma_1\,t^{1-\alpha}+Q+o(t^{1-\alpha});
\]
\item if $\alpha=1$,
\[
\gamma(t)=at-\mathcal{M}^{-1}\nabla U(a)\,\log t+Q'+o(1);
\]
\item if $\alpha\in(1/2,1)$,
\[
\gamma(t)
=
at+\Gamma_1\,t^{1-\alpha}
-\frac{\mathcal{M}^{-1}\nabla^2U(a)\Gamma_1}{2\alpha(1-2\alpha)}\,t^{1-2\alpha}
+Q''+o(t^{1-2\alpha}).
\]
\end{enumerate}
\end{thm}

\begin{rem}
In Theorem \ref{thm:hyp_alpha_expansive_1} we include the Newtonian case $\alpha=1$, already treated in \cite{PolimeniTerracini}, to emphasize the continuity of the asymptotic structure across the homogeneity exponent.
\end{rem}

To describe hyperbolic solutions when $\alpha\in(0,1/2)$, for every $a\in\Omega$ we define a sequence of correction vectors $\Gamma_k=\Gamma_k(\alpha,a)\in\mathbb{R}^{dN}$ for
\[
k=2,\dots,\left\lfloor\frac{1}{2\alpha}\right\rfloor+1,
\]
recursively by
\begin{equation}\label{eq:gamma_k}
\Gamma_k
:=
-\frac{
\displaystyle
\sum_{q=1}^{k-1}\frac{1}{q!}
\sum_{j_1+\cdots+j_q=k-1}
\mathcal{M}^{-1}\nabla^{q+1}U^\alpha(a)\big[\Gamma_{j_1},\dots,\Gamma_{j_q}\big]
}{
k\alpha\,(1-k\alpha)
},
\end{equation}
where $\Gamma_1$ is given by \eqref{e:Gamma1} (in particular, $\Gamma_2=-\frac{\mathcal{M}^{-1}\nabla^2U(a)\Gamma_1}{2\alpha(1-2\alpha)}$ as in Theorem~\ref{thm:hyp_alpha_expansive_1}).

In the borderline case $\alpha=1/2$ we set
\[
\widetilde\Gamma=\widetilde\Gamma(a):=-4(\mathcal{M}^{-1})^2\nabla^2U(a)\,\nabla U(a).
\]

\begin{thm}[Hyperbolic motions with {$\alpha\in(0,1/2]$}]
\label{thm:hyp_alpha_expansive_2}
Let $x\in\mathcal{X}$ be arbitrary and let $\alpha\in(0,1/2]$. For any $a\in\Omega$ there exists a hyperbolic solution $\gamma$ of \eqref{eq:NBP} with $\gamma(1)=x$ such that, as $t\rightarrow+\infty$:
\begin{enumerate}
\item if $\alpha=1/2$, then
\begin{equation}\label{eq:hyp_small_alpha_2}
\gamma(t)=at+\Gamma_1\,t^{1/2}+\widetilde\Gamma\log t+o(\log t);
\end{equation}
\item if $\alpha\in(0,1/2)$, then
\begin{equation}\label{eq:hyp_small_alpha_1}
\gamma(t)
=
at+\sum_{k=1}^{P}\Gamma_k\,t^{1-k\alpha}+o\!\left(t^{1-P\alpha}\right),
\end{equation}
where $P=\lfloor 1/(2\alpha)\rfloor+1$.
\end{enumerate}
\end{thm}

\medskip

For $\alpha\in(0,2)$, let $b_m\in\mathcal{E}:=\{x\in\mathcal{X}:\|x\|_{\mathcal{M}}^2=1\}$ be a minimal central configuration for $U^\alpha$, namely,
\[
U^\alpha(b_m)=\min_{\mathcal{E}}U^\alpha.
\]
Set
\begin{equation}\label{e:beta}
\beta:=\sqrt[2+\alpha]{\frac{(2+\alpha)^2}{2}\,U^\alpha(b_m)}.
\end{equation}
Then, $r_0(t)=\beta b_m t^{\frac{2}{2+\alpha}}$ is a homothetic solution of \eqref{eq:NBP} (see also \cite{BarutelloFerrarioTerracini2008}).

\begin{thm}[Pure parabolic motions]\label{thm:par_alpha}
Let $\alpha\in(0,2)$. For any $x\in\mathcal{X}$ and any minimal normalized central configuration $b_m$, there exists a parabolic solution $\gamma$ of \eqref{eq:NBP} with $\gamma(1)=x$ of the form
\begin{equation}\label{eq:parabolic}
\gamma(t)=\beta b_m t^{\frac{2}{2+\alpha}}+O\!\left(t^{\frac{\alpha}{2+\alpha}}\right),
\qquad\text{as }t\to+\infty,
\end{equation}
where $\beta$ is given by \eqref{e:beta}.
\end{thm}

\begin{rem}
The range $\alpha\in(0,2)$ in Theorem~\ref{thm:par_alpha} reflects the fact that our construction works as a low-order perturbation of the reference homothetic path $r_0(t)$.
If $\alpha\ge 2$, the associated reference growth is no longer suitable for the variational perturbation scheme.
\end{rem}

\medskip

To introduce hyperbolic-parabolic expansive solutions, fix $a\in\Delta\setminus\{0\}$ and define $i\sim_a j$ if and only if $a_i=a_j$. This induces a partition into clusters, and we denote by $K$ a generic cluster. For each cluster $K$, let $b^K$ be a minimal central configuration for the restricted potential
\[
U^\alpha_K(x):=\sum_{i<j,\ i,j\in K}\frac{m_i m_j}{|x_i-x_j|^\alpha},
\]
and set
\[
\beta^K:=\sqrt[2+\alpha]{\frac{(2+\alpha)^2}{2}\,U^\alpha_K(b^K)}.
\]
Collecting these into vectors (over the cluster decomposition) we write
\begin{equation}\label{e:b_a_cluster}
b_m=(b^{K_1},b^{K_2},\dots),
\qquad
\beta=(\beta^{K_1},\beta^{K_2},\dots),
\end{equation}
and we call $b_m$ a minimal central configuration of the $a$-clustered potential.

\begin{thm}[hyperbolic-parabolic motions]\label{thm:HP_alpha}
Let $\alpha\in(1/2,2)$. For any $x\in\mathcal{X}$, any collision configuration $a\in\Delta$ and any normalized minimal central configuration $b_m$ of the $a$-clustered potential, there exists a hyperbolic-parabolic motion $\gamma$ of \eqref{eq:NBP} with $\gamma(1)=x$ of the form
\begin{equation}\label{eq:hyperbolic_parabolic}
\gamma(t)=at+\beta b_m t^{\frac{2}{2+\alpha}}+O(t^\delta),
\qquad\text{as }t\to+\infty,
\end{equation}
where
\[
\delta=\max\left\{1-\alpha,\frac{\alpha}{2+\alpha}\right\}.
\]
Equivalently, $\delta=1-\alpha$ if $\alpha\in(1/2,\sqrt{3}-1)$ and $\delta=\frac{\alpha}{2+\alpha}$ if $\alpha\in(\sqrt{3}-1,2)$.
\end{thm}

\begin{rem}
For $\alpha=1$, estimates of the type \eqref{eq:parabolic} improve the remainder bounds obtained in \cite{PolimeniTerracini} for pure parabolic motions.
\end{rem}

We emphasize that in this work we focus on \emph{half-entire} expansive motions, namely trajectories $\gamma(t)$ defined on a half-line $[t_0,+\infty)$ (and, by time reversal, on $(-\infty,t_0]$). The existence of expansive solutions defined for all $t\in\mathbb{R}$ remains an open problem. For other problems related to the $N$-body problem we refer, e.g., to \cite{MR1610784,paradela2022oscillatory}.

\subsection{Jacobi-Maupertuis geometry, Hamilton-Jacobi theory, and the boundary at infinity.}
The variational nature of our construction admits a natural geometric and dynamical interpretation through the Jacobi-Maupertuis (JM) metric and the associated Hamilton-Jacobi equation. 
For each energy level $h\ge 0$, the configuration space $\Omega$ can be endowed with the conformal Riemannian metric
\[
j_h^\alpha = 2\bigl(h + U^\alpha(x)\bigr)\, g_{\mathcal M},
\]
whose geodesics, after a suitable reparameterization, coincide with solutions of \eqref{eq:NBP} at energy $h$. In particular, free-time minimizers of the Lagrangian action correspond to geodesic rays for the JM distance, and, as shown in \cite{PolimeniTerracini}, expansive solutions minimizing the renormalized action functional therefore provide a distinguished class of such rays.

From a dynamical viewpoint, geodesic rays are the fundamental objects underlying the description of the boundary at infinity of $(\Omega, j_h^\alpha)$, either in the sense of Gromov or via Busemann (horofunction) compactifications. In the Newtonian case $\alpha=1$, it is known that collisionless asymptotic directions determine boundary points in the hyperbolic case, while parabolic motions associated with minimal central configurations give rise to distinguished Busemann functions. More generally, geodesic rays at nonnegative energy are necessarily expansive and exhibit a decomposition into linearly diverging cluster centers and internal parabolic dynamics.

Within this framework, the expansive solutions constructed in Theorems~\ref{thm:hyp_alpha_expansive_1}--\ref{thm:HP_alpha} can be interpreted as explicit geodesic rays of the JM metric for homogeneous potentials of degree $-\alpha$. Their asymptotic invariants -- namely the hyperbolic direction $a$, the parabolic limit shape, and the cluster decomposition -- provide a natural system of coordinates for the boundary at infinity, extending to the homogeneous setting the geometric picture previously known in the Newtonian case. It has to be noticed, however, that in the hyperbolic-parabolic case, two of rays associated with the same cluster decomposition do not keep a finite JM distance and hence they do not identify the same point of Gromov's boundary.

At the same time, this construction is closely related to the theory of viscosity solutions of the stationary Hamilton--Jacobi equation
\begin{equation}\label{eq:HJ_intro}
    H(x,\nabla u(x)) = h,
\end{equation}
where $H(x,p)=\frac12\|p\|_{\mathcal M^{-1}}^2 - U^\alpha(x)$ is the Hamiltonian of the system. In the weak KAM framework, global viscosity solutions of \eqref{eq:HJ_intro} can be obtained through the Lax-Oleinik semigroup, and are characterized by the existence of \emph{calibrating curves}, i.e.\ curves that realize equality in the variational principle
\[
u(\gamma(t)) - u(\gamma(s)) = \int_s^t \bigl( L(\gamma(\tau),\dot\gamma(\tau)) + h \bigr)\, \ud\tau,
\qquad t\ge s.
\]

The expansive motions obtained in this paper are precisely such calibrating curves: they are free-time minimizers of the action and hence generate global viscosity solutions of \eqref{eq:HJ_intro}. In particular, as already shown in \cite{BertiPolimeniTerracini} in the case $\alpha=1$, different dynamical regimes correspond to distinct classes of solutions:
\begin{itemize}
\item hyperbolic motions yield viscosity solutions with prescribed asymptotic linear drift, corresponding to Busemann-type functions associated with collisionless directions;
\item parabolic motions associated with minimal central configurations generate solutions whose calibrating curves converge, after rescaling, to central configurations;
\item hyperbolic-parabolic motions give rise to viscosity solutions encoding a mixed asymptotic structure, reflecting the cluster decomposition of the dynamics.
\end{itemize}

In this sense, expansive solutions provide a dynamical realization of characteristics of global solutions to the Hamilton-Jacobi equation, and the refined asymptotic expansions established in this paper translate into precise information on the structure of such solutions at infinity. Finally, we note that this boundary picture is intrinsically tied to the nonnegative energy regime: for negative energies, the Jacobi-Maupertuis metric has finite diameter and does not admit geodesic rays, so that no nontrivial boundary at infinity arises.

\subsection{Outline of the paper}
In Section~\ref{sec:variational_setting} we introduce the variational framework used to prove Theorems~\ref{thm:hyp_alpha_expansive_1}, \ref{thm:hyp_alpha_expansive_2}, \ref{thm:par_alpha} and \ref{thm:HP_alpha}, and we describe the functional space $\mathcal{D}_0^{1,2}(1,+\infty)$ in which the minimization procedure is carried out.
Section~\ref{sec:existence} is devoted to the existence part, namely the construction of minimizers of the renormalized action by the direct method in the calculus of variations.
Finally, in Section~\ref{sec:asymptotic_estimates} we derive refined asymptotic estimates for the corresponding expansive motions, completing the proof of Theorems~\ref{thm:hyp_alpha_expansive_1}, \ref{thm:hyp_alpha_expansive_2}, \ref{thm:par_alpha} and \ref{thm:HP_alpha}.


\section{The variational setting}\label{sec:variational_setting}

In this section, we adapt, to our setting, the variational framework developed in \cite{PolimeniTerracini} to establish the existence of hyperbolic, parabolic, and hyperbolic-parabolic motions. This approach is based on the minimization of a renormalized Lagrangian action. Since we consider a homogeneous Newtonian potential, the renormalization of the action is not always required and depends on the value of the homogeneity parameter $\alpha$. Accordingly, we distinguish between different definitions of the action used in the minimization arguments, depending on $\alpha$ and on the type of expansive motion under consideration.

For the $N$-body problem, the Hamiltonian $H$ is defined on $\Omega \times \mathbb{R}^{dN}$ by $H(x, p) = \frac{1}{2} \|p\|_{\mathcal{M}^{-1}}^2 - U^\alpha(x)$, where $\|\cdot\|_{\mathcal{M}^{-1}}$ denotes the dual norm associated with the mass inner product. The corresponding Lagrangian is given by
\begin{equation}\label{eq:lagrangian}
    L(x, v) = \frac{1}{2} \|v\|_\mathcal{M}^2 + U^\alpha(x).
\end{equation}

Given two configurations $x, y \in \mathcal{X}$ and $T > 0$, we consider the set of admissible paths
\[
\mathcal{C}(x, y, T) = \left\{ \gamma : [a, b] \rightarrow \mathcal{X} \,:\, \gamma \text{ is absolutely continuous},\, \gamma(a) = x,\, \gamma(b) = y,\, b - a = T \right\},
\]
as well as the free-time path space
\[
\mathcal{C}(x, y) = \bigcup_{T > 0} \mathcal{C}(x, y, T).
\]

For any curve $\gamma \in \mathcal{C}(x, y, T)$, the Lagrangian action is defined by
\[
    \mathcal{A}_L(\gamma) = \int_{a}^{b} L(\gamma, \dot{\gamma})\, \mathrm{d}t 
    = \int_{a}^{b} \left( \frac{1}{2} \|\dot{\gamma}\|_\mathcal{M}^2 + U^\alpha(\gamma) \right) \mathrm{d}t.
\]

By Hamilton's Principle of Least Action, if a curve $\gamma \in \mathcal{C}(x, y, T)$ minimizes the Lagrangian action, then it satisfies Newton's equations at every time $t \in [a, b]$ such that $\gamma(t)$ is collision-free. Nevertheless, since there exist curves with isolated collisions and finite action (see \cite{Poincare_SolutionsPeriodiques}), action-minimizing paths do not automatically correspond to genuine solutions. Marchal's Principle (Theorem \ref{thm_marchal}) overcomes this difficulty by guaranteeing that minimizers are collision-free in the interior of the time interval, thus justifying the variational approach.

The strategy of proving Marchal's Principle through averaged variations was originally introduced by Marchal \cite{Marchal_MethodOfMinimization}, and later refined and completed by Chenciner \cite{Chenciner}, and Ferrario and Terracini \cite{FerrarioTerracini} (see also \cite{BarutelloFerrarioTerracini2008} for a study on Sundman-type asymptotic estimates for collision solutions).

\begin{thm}[Marchal \cite{Marchal_MethodOfMinimization}, Chenciner \cite{Chenciner}, Ferrario and Terracini \cite{FerrarioTerracini}]\label{thm_marchal}
Assume $d\geq 2$, let $x, y \in \mathcal{X}$, and let $\gamma \in \mathcal{C}(x, y)$ be defined on some interval $[a, b]$. If
\[
    \mathcal{A}_L(\gamma) = \min \left\{ \mathcal{A}_L(\sigma)\ :\ \sigma \in \mathcal{C}(x, y, b-a) \right\},
\]
then $\gamma(t) \in \Omega$ for all $t \in (a, b)$.
\end{thm}

To state the Renormalized Action Principle, we introduce the following notion:

\begin{defn}\label{def:free-time_minimizers}
A curve $\gamma: I \to \mathcal{X}$ is a \emph{free-time minimizer} for the Lagrangian action at energy $h$ if, for all intervals $[a, b], [a', b'] \subset I$ and all curves $\sigma: [a', b'] \to \mathcal{X}$ with $\gamma(a) = \sigma(a')$ and $\gamma(b) = \sigma(b')$, we have:
\[
    \int_a^b L(\gamma, \dot{\gamma})\, \mathrm{d}t + h(b - a) \leq \int_{a'}^{b'} L(\sigma, \dot{\sigma})\, \mathrm{d}t + h(b' - a').
\]
\end{defn}

The main strategy to prove the existence of expansive motions for the $N$-body problem, as free-time minimizers, is to consider solutions of the form:
\[
    \gamma(t) = r_0(t) + \varphi(t) + x - r_0(1),
\]
where $r_0$ is a fixed reference path, $\varphi$ is a perturbation in a suitable function space, and $x$ is the initial configuration. 

Specifically, we work with the space of perturbations:
\begin{displaymath}\label{eq:space}
    \mathcal{D}_0^{1,2}(1,+\infty) = \left\{ \varphi \in H^1_{\mathrm{loc}}([1, +\infty), \mathcal{X}) : \varphi(1) = 0,\ \int_1^{+\infty} \|\dot{\varphi}(t)\|_\mathcal{M}^2\, \mathrm{d}t < +\infty \right\},
\end{displaymath}
endowed with the norm:
\[
    \|\varphi\|_{\mathcal{D}} = \left( \int_1^{+\infty} \|\dot{\varphi}(t)\|_\mathcal{M}^2\, \mathrm{d}t \right)^{1/2}.
\]

\begin{prop}[Cfr. \cite{BDFT}]
The space $\mathcal{D}_0^{1,2}(1,+\infty)$ is a Hilbert space, and $C_c^\infty(1, +\infty)$ is dense in it.
\end{prop}

\begin{prop}[\textit{Hardy inequality}, Cfr. \cite{BDFT}]\label{dis_hardy}
For every $\varphi \in \mathcal{D}_0^{1,2}(1, +\infty)$, the following inequalities hold:
\begin{displaymath}
    \int_1^{+\infty} \frac{\|\varphi(t)\|_\mathcal{M}^2}{t^2}\, \mathrm{d}t \leq 4 \int_1^{+\infty} \|\dot{\varphi}(t)\|_\mathcal{M}^2\, \mathrm{d}t,
\end{displaymath}
\begin{equation}\label{dis_space_D012}
    \sup_{t \in (1,+\infty)} \frac{\|\varphi(t)\|_\mathcal{M}^2}{t - 1} \leq \int_1^{+\infty} \|\dot{\varphi}(t)\|_\mathcal{M}^2\, \mathrm{d}t.
\end{equation}
\end{prop}

Proposition \ref{dis_hardy} can also be extended as the following proposition, which shows that the space $\mathcal{D}_0^{1,2}(1,+\infty)$ is compactly embedded in a space $L^2(1,+\infty)$ equipped with proper weights. 
We define, for all $\varepsilon\ge 0$, the space $L^2(1,+\infty;\ud t/t^{2+\varepsilon})$ as the space of functions $\varphi$ such that
\[
\int_1^{+\infty} \frac{\|\varphi(t)\|_{\mathcal{M}}^2}{t^{2+\varepsilon}}\, \ud t < +\infty.
\]

\begin{prop}[Cfr. Berti, Polimeni, Terracini 2025 \cite{BertiPolimeniTerracini}]
\label{prop:compact_emb}
    For all $\varepsilon \ge0$ and for all $\varphi\in\mathcal{D}_0^{1,2}(1,+\infty)$, the following Hardy-type inequality holds
    \begin{equation}\label{hardy_compact}
        \int_{1}^{+\infty} \frac{\|\varphi(t)\|_\mathcal{M}^2}{t^{2+\varepsilon}}\ \ud t \leq \frac{4}{(1+\varepsilon)^2} \int_{1}^{+\infty} \|\dot{\varphi}(t)\|_\mathcal{M}^2\ \ud t,
    \end{equation}
    that is, the space $\mathcal{D}_0^{1,2}(1,+\infty)$ is continuously embedded in the space $L^2(1,+\infty;\ud t/t^{2+\varepsilon})$. Besides, $\mathcal{D}_0^{1,2}(1,+\infty)$ is compactly embedded in the space $L^2(1,+\infty; \ud t/t^{2+\varepsilon})$ for all $\varepsilon>0$.
\end{prop}

\noindent{\bf The variational problem in the space of perturbations.}
As mentioned earlier, we study expansive motions with a reference path $r_0(t)$ which depends on the subclass of motion under consideration and, in the hyperbolic case, on the homogeneity parameter $\alpha$.

More precisely, for an expansive motion $\gamma:[1,+\infty)\to\mathcal{X}$ of the form
\[
\gamma(t) = r_0(t) + \varphi(t) + x - r_0(1),
\]
with $\varphi \in \mathcal{D}_0^{1,2}(1,+\infty)$ and $x \in \mathcal{X}$, we define the reference paths as follows.

\smallskip
\noindent $\bullet$ {\em Hyperbolic solutions.}
For $a \in \Omega$, we set
\begin{equation}
    \label{eq:hyperbolic_r_0}
    r_0(t)=r_0^H(t):=at+\sum_{k=1}^{\lfloor 1/(2\alpha)\rfloor}\Gamma_k(a)\,t^{1-k\alpha},
\end{equation}
where $\Gamma_k(a)$ are defined in \eqref{eq:gamma_k}. Note, if $\alpha > 1/2$, then $r_0(t)=at$.

\noindent $\bullet$ {\em Parabolic solutions.}
For a minimal central configuration $b_m$ and $\beta$ as in \eqref{e:beta}, we define
\begin{equation}
    \label{eq:parabolic_eq}
    r_0(t)=r_0^P(t):= \beta \, b_m \,t^\frac{2}{2+\alpha}.
\end{equation} 

\noindent $\bullet$ {\em Hyperbolic-Parabolic solutions.}
For $a \in \Delta \setminus\{0\}$, $b_m$ and $\beta$ as in \eqref{e:b_a_cluster}, we define
\begin{equation}
    \label{eq:HP_eq}
    r_0(t)=r_0^{HP}(t):=at+ \beta \, b_m \,t^\frac{2}{2+\alpha}.
\end{equation}

As in the case $\alpha=1$, studied in \cite{PolimeniTerracini}, for certain values of $\alpha$ the Lagrangian of an expansive motion in the $N$-body problem is not integrable on the half-line $[1,+\infty)$, since it may grow at most linearly at infinity. To address this lack of integrability on $[1,+\infty)$, we introduce the following renormalization of the action.

\begin{defn}[Renormalized Lagrangian Action]\label{def:ren_action}
Given $\alpha$, an initial configuration $x \in \mathcal{X}$ and a reference path $r_0$, 
we define the \emph{Renormalized Lagrangian Action} as the functional $\mathcal{A}_x^{ren}: \mathcal{D}_0^{1,2}(1,+\infty) \to \mathbb{R}$ given by:
\begin{displaymath}\label{def:renormalized_action}
    \mathcal{A}_x^{ren}(\varphi) = \int_1^{+\infty}  \frac{1}{2} \|\dot{\varphi}(t)\|_\mathcal{M}^2 + U^\alpha(r_0(t) +\varphi(t) +  x - r_0(1)) - U^\alpha(r_0(t)) - \langle \mathcal{M} \ddot{r}_0(t), \varphi(t) \rangle \, \mathrm{d}t.
\end{displaymath}
\end{defn}

It is trivial to see that the renormalization is not required when considering hyperbolic motions with $\alpha\in(1,+\infty)$, as, for $\varphi\in\mathcal{D}_0^{1,2}(1,+\infty)$,  $U(at + \varphi(t) +x - r_0(1))\in L^1([1,+\infty))$. In such case, it will be enough to consider the usual Lagrangian action $\mathcal{A}_x:\mathcal{D}_0^{1,2}(1,+\infty)\rightarrow\R$,
\[
\mathcal{A}_x(\varphi) = \int_1^{+\infty}  \frac{1}{2} \|\dot{\varphi}(t)\|_\mathcal{M}^2 + U^\alpha(r_0(t) +\varphi(t) +  x - r_0(1)) - \langle \mathcal{M} \ddot{r}_0(t), \varphi(t) \rangle \, \mathrm{d}t. 
\]

To simplify the notation, in what follows we will use $\mathcal{A}$ to denote both the Lagrangian action and its renormalized version. The distinction between the two cases will be clear from the context, and the dependence on $x$ will be omitted. We will omit also the apex $\alpha$ on $U^\alpha$.

\begin{rem}

In the hyperbolic-parabolic setting, given a configuration $a\in\Delta$, the $a$-cluster partition of the bodies can be exploited for the following decomposition of the renormalized Lagrangian action, which will be used several times in the paper (see \cite{PolimeniTerracini}):
\begin{equation}\label{eq:decomposed_action}
\begin{split}
    \mathcal{A}(\varphi) & = \sum_{K\in\mathcal{P}} \mathcal{A}_{K}(\varphi) + \sum_{K_1,K_2\in\mathcal{P},\ K_1\neq K_2} \mathcal{A}_{K_1,K_2}(\varphi)\\
    &  = \sum_{K\in\mathcal{P}}\bigg(\sum_{i,j\in K,\ i<j} m_i m_j\mathcal{A}_{K}(\varphi)_{ij}\bigg) + \frac{1}{2}\sum_{K_1,K_2\in\mathcal{P},\ K_1\neq K_2}\bigg(\sum_{i\in K_1,\ j\in K_2} m_i m_j\mathcal{A}_{K_1,K_2}(\varphi)_{ij}\bigg),
\end{split}
\end{equation}
where
\[
\begin{split}
    \mathcal{A}_{K}(\varphi)_{ij} = & \int_{1}^{+\infty} \frac{1}{2M_{K}} |\dot{\varphi}_{ij}(t)|^2 + \frac{1}{|\beta^K b^K_{ij}t^{\frac{2}{2+\alpha}}+\varphi_{ij}(t)+\Tilde{x}_{ij}|^\alpha} - \frac{1}{|\beta^K b^K_{ij}t^{\frac{2}{2+\alpha}}|^\alpha} \\
    &+ \frac{2\alpha}{(2+\alpha)^2}\frac{\beta^K}{M_K} \frac{\langle b^K_{ij},\varphi_{ij}(t)\rangle}{t^{\frac{2(1+\alpha)}{2+\alpha}}}\ \ud t,
\end{split}
\]
and
\[
\begin{split}
    \mathcal{A}_{K_1,K_2}(\varphi)_{ij} = &\int_{1}^{+\infty}\frac{1}{2M_{K_{1,2}}} |\dot{\varphi}_{ij}(t)|^2+ \frac{1}{|a_{ij}t+\beta^{K_{1,2}} b^{K_{1,2}}_{ij}t^{\frac{2}{2+\alpha}}+\varphi_{ij}(t)+\Tilde{x}_{ij}|^\alpha} - \frac{1}{|a_{ij}t|^\alpha}\ \ud t.
\end{split}
\]
We adopt the notation:
\begin{displaymath}
    \begin{split}
        &M_K = \sum_{i\in K}m_i,\quad M_{K_{1,2}} = \sum_{i\in K_1\cup K_2}m_i,\\
        &\beta^{K_{1,2}}=(\beta^{K_1},\beta^{K_2}),\quad b^{K_{1,2}}=(b^{K_1},b^{K_2}),\\
        &\beta^{K_{1,2}}b^{K_{1,2}}=(\beta^{K_1}b^{K_1},\beta^{K_2}b^{K_2}).
    \end{split}
\end{displaymath}

The term $\mathcal{A}_K$ refers to the motion of the bodies inside each cluster $K$, and the term $\mathcal{A}_{K_1,K_2}$ refers to the interactions between pairs of bodies that belong to pairs of clusters $K_1\ne K_2$.

{We also point out that in the decomposition above, to simplify our computations even more, we use a slight variation of our usual renormalization term $-U(r_0(t))$, which can be replaced by
    \[
    -\tilde{U}(r_0(t)) = -\sum_{K\in\mathcal{P}} \bigg(\sum_{i,j\in K,\ i<j} \frac{m_i m_j}{|\beta^K b_{ij}^K t^{\frac{2}{2+\alpha}}|^\alpha}\bigg) - \frac{1}{2}\sum_{K_1,K_2\in\mathcal{P},\ K_1\neq K_2} \bigg(\sum_{i\in K_1,\ j\in K_2} \frac{m_i m_j}{|a_{ij}t|^\alpha}\bigg).
    \]}
\end{rem}
This will not affect our search for minimizers, since the renormalization term does not depend on $\varphi$ (see \cite[Section 5]{PolimeniTerracini}).

For some $\varphi \in \mathcal{D}^{1,2}_0(1,+\infty)$ and for a pair of indexes $i<j$, $r_0(t)_{ij}+\varphi_{ij}(t)+x_{ij}-r_0(1)_{ij}$ may be $0$ at some instant, in a way that makes $U$ non integrable. As a consequence, for some $\varphi \in \mathcal{D}^{1,2}_0(1,+\infty)$, $\mathcal{A}(\varphi)$ might not be finite. The next proposition informs us that $\mathcal{A}(\varphi)\in \mathbb R$ if $r_0(t)+\varphi(t)+x-r_0(1)$ does not admit collisions (except, at most, at the initial configuration).

\begin{prop}[Finiteness of $\mathcal{A}$ over non-collision sets]
\label{prop:well-posedness}
Assume either $\alpha >0$ and 
$r_0$ as in \eqref{eq:hyperbolic_r_0}, or $\alpha \in (0,2)$ and $r_0$ as in \eqref{eq:parabolic_eq}, or $\alpha \in (1/2,2)$ and $r_0$ as in \eqref{eq:HP_eq}. In all these settings, for any fixed initial configuration $x\in\mathcal{X}$, $\mathcal A$ is finite over the set
\[
\mathcal{D}_0^{1,2}(1,+\infty) \setminus \{\varphi\in\mathcal{D}_0^{1,2}(1,+\infty):\exists i<j,\ \exists t\in (1,+\infty)  \text{ such that } r_0(t)_{ij}+\varphi_{ij}(t)+x_{ij}-r_0(1)_{ij}=0\}.
\]
\end{prop}

\begin{proof}
First, we note that the cases corresponding to $\alpha = 1$ were already treated in \cite{PolimeniTerracini}. We now consider the remaining cases one at a time. 

\noindent $\bullet$ The case $\alpha >1$ and $r_0(t)$ hyperbolic \eqref{eq:hyperbolic_r_0}, the statement is trivial, since
\[
U(r_0(t)+\psi(t)) = \frac{U(z(t))}{\|at + \psi(t)\|_\mathcal{M}^\alpha} = \frac{U(\hat a)}{\|a\|_\mathcal{M}^\alpha t^\alpha}+\mbox{ lower order term} \in L^1(1,+\infty),
\]
since $\alpha >1$.

\noindent $\bullet$
The case $\alpha \in (1/2,1)$, $r_0$ hyperbolic \eqref{eq:hyperbolic_r_0}, one can argue as in the case $\alpha=1$ (\cite{PolimeniTerracini}) to decompose $\mathcal{A}$ as follows:
\[
\mathcal{A}(\varphi)= \sum_{i\neq j}m_im_j \int_1^{+\infty} \frac{|\dot\varphi_{ij}(t)|^2}{2}+\frac{1}{|a_{ij}t + \varphi_{ij}(t) + x_{ij} - a_{ij}|^\alpha}-\frac{1}{|a_{ij}t|^\alpha}\,\ud t.
\]
It suffices then to observe that, for every pair $i\neq j$, we have, by Taylor's expansion,
\[
\frac{1}{|a_{ij}t + \varphi_{ij}(t) + x_{ij} - a_{ij}|^\alpha} - \frac{1}{|a_{ij}t|^\alpha} \le \int_{0}^{1}\alpha\frac{|\varphi_{ij}(t) + x_{ij} - a_{ij}|}{|a_{ij}t + s(\varphi_{ij}(t) + x_{ij} - a_{ij})|^\alpha}\ \ud s \le \frac{C}{t^{\frac{1}{2}+\alpha}}\in L^1(1,+\infty),
\]
for a proper $C\in\R$.

\smallskip
\noindent $\bullet$ The case $\alpha \in (0,1/2]$ and $r_0$ hyperbolic \eqref{eq:hyperbolic_r_0} is more subtle. We argue as follows. 

{Consider
\[
r_0(t)=at+\Gamma_1\,t^{1-\alpha}+\Gamma_2\,t^{1-2\alpha}+\dots+\Gamma_{m-1}\,t^{1-(m-1)\alpha}+\Gamma_m\,t^{1-m\alpha},
\]
where $m=\lfloor 1/(2\alpha)\rfloor$.

On one hand, we have 
\[
\begin{aligned}
\mathcal M \ddot r_0(t)&=-\alpha(1-\alpha) \mathcal M \,\Gamma_1\, t^{-(1+\alpha)}-2\alpha(1-2\alpha)\mathcal M\Gamma_2\,t^{-(1+2\alpha)}+\dots-m\alpha(1-m\alpha)\mathcal M\Gamma_m\,t^{-(1+\alpha m)}
\\
&=-\frac{\alpha}{t^{1+\alpha}}\left\{(1-\alpha)\mathcal M \Gamma_1 +2(1-2\alpha)\mathcal M\Gamma_2\,t^{-\alpha} +\dots+m(1-m\alpha)\mathcal M \Gamma_m\,t^{-(m-1)\alpha}\right\}
\\
&=-\frac{\alpha}{t^{1+\alpha}}\left\{\mathcal S_0+\mathcal S_1\,t^{-\alpha}+\mathcal S_2\,t^{-2 \alpha}+\dots + \mathcal S_{m-1}\,t^{-(m-1)\alpha}\right\},
\end{aligned}
\]
where
\[
\mathcal S_j:=(j+1)(1-(j+1)\alpha)\mathcal M \Gamma_{j+1} \ \mbox{ for } \ j=0,\dots,m-1.
\]
On the other hand,
\[
\begin{aligned}
\nabla U(r_0(t))&=\nabla U(t(a+\Gamma_1 t^{-\alpha}+\Gamma_2\,t^{-2\alpha}+\dots+\Gamma_m\,t^{-m\alpha}))
\\
&=\frac{1}{t^{1+\alpha}} \nabla U(a+\Gamma_1 t^{-\alpha}+\Gamma_2\,t^{-2\alpha}+\dots+\Gamma_m\,t^{-m\alpha})
\\
&=\frac1{t^{1+\alpha}}\left\{\nabla U(a)+\nabla^2 U(a)\left(\Gamma_1 t^{-\alpha}+\Gamma_2\,t^{-2\alpha}+\dots+\Gamma_m\,t^{-m\alpha}\right)\right.
\\
&\quad + \frac{1}{2}\nabla^3 U(a)\left[\Gamma_1,\Gamma_1 \right]t^{-2\alpha}+\frac12\nabla^3 U(a)\left[\Gamma_1,\Gamma_2\right]\,t^{-3\alpha}+\dots+\frac12\sum_{i,j} \nabla^3 U(a)\left[\Gamma_i,\Gamma_j\right]\,t^{-(i+j)\alpha}
\\
&\quad + \frac{1}{3!}\sum_{i,j,k}\nabla^4 U(a)\left[\Gamma_i,\Gamma_j,\Gamma_k\right]\,t^{-(i+j+k)\alpha}+\dots + \frac{1}{q!}\sum_{i_1,\dots,i_q}\nabla^{q+1}U(a)\left[\Gamma_{i_1},\dots,\Gamma_{i_q}\right]\,t^{-(i_1+\dots+i_q)\alpha}
\\
&\quad \left.+ \frac{1}{p!}\nabla^{p+1}U(a)\left[\Gamma_1,\dots,\Gamma_1\right]\,t^{-m\alpha}+o(t^{-m\alpha})\right\}
\\
&=\frac{1}{t^{1+\alpha}}\left\{\nabla U(a)+\mathcal R_1\,t^{-\alpha}+\mathcal R_2\,t^{-2\alpha}+\dots+\mathcal R_{m-1}\,t^{-(m-1)\alpha}+\mathcal R_m\,t^{-m\alpha}+o(t^{-m\alpha}) \right\},
\end{aligned}
\]
where the last two identities are true as $t \rightarrow+\infty$, and where
\begin{equation} 
\label{eq:R}
\mathcal R_j:=\sum_{q=1}^{j}\frac{1}{q!}\sum_{i_1,\dots,i_q}\nabla^{q+1}U(a)\left[\Gamma_{i_1},\dots, \Gamma_{i_q}\right], \ \mbox{ for } \ j=1,\dots,m.
\end{equation}
Hence,
\[
\begin{aligned}
\mathcal M \ddot r_0(t)-\nabla U(r_0(t))&=\frac{1}{t^{1+\alpha}}\left\{-\left(\alpha \mathcal S_0+\nabla U(a)\right)-\left(\alpha \mathcal S_1+\mathcal R_1\right)t^{-\alpha}-\left(\alpha \mathcal S_2+\mathcal R_2\right)t^{-2\alpha}\right.
\\
&\quad \left.+\dots-\left(\alpha \mathcal S_{m-1}+\mathcal R_{m-1}\right)t^{-(m-1)\alpha}-\mathcal R_m\,t^{-m\alpha}+o(t^{-m\alpha}) \right\}.
\end{aligned}
\]
Under the very definition \eqref{eq:gamma_k} of $\Gamma_k=\Gamma_k(a)$, we get
\[
\alpha \mathcal S_0+\nabla U(a)=\alpha \mathcal S_j+\mathcal R_j=0, \ \mbox{ for } \ j=1,\dots,m-1,
\]
and hence
\begin{equation}
    \label{e:diff_r0}
    \mathcal M \ddot r_0(t)-\nabla U(r_0(t))=\mathcal R_m\,t^{-1-(m+1)\alpha }+o(t^{-1-(m+1)\alpha }),
\end{equation}
as $t\rightarrow+\infty$.
}

As a consequence, for every $\varphi \in \mathcal{D}_0^{1,2}(1,+\infty)$, we have
\[
\begin{split}
    &U(r_0(t)+\varphi(t))-U(r_0(t))-\langle \ddot r_0(t), \varphi(t)\rangle_\mathcal{M}\\
    & = U(r_0(t)+\varphi(t))-U(r_0(t))-\langle \nabla U( r_0(t)), \varphi(t)\rangle  - \langle \mathcal R_m\,t^{-1-(m+1)\alpha }+o(t^{-1-(m+1)\alpha }), \varphi(t)\rangle.
\end{split}
\]

We notice that 
\[
\left\|\langle \mathcal R_m t^{-1-\alpha(m+1)}, \varphi(t)\rangle \right\|=
\left\| \mathcal R_m\right\|t^{-1-(m+1)\alpha}\,\| \varphi(t)\|\le \frac{C}{t^{\frac12+(m+1)\alpha}} \in L^1(1,+\infty),
\]
for some $C\in\R$.

In addition:
\[
\begin{split}
    U(r_0(t)+\varphi(t))-U(r_0(t))-\langle \nabla U( r_0(t)), \varphi(t)\rangle & = \int_{0}^{1} \langle \nabla U(r_0(t)+s \varphi(t))- \nabla U( r_0(t)), \varphi(t)\rangle\ \ud s\\
    & = \int_{0}^{1}\int_{0}^{1} \langle \nabla^2 U(r_0(t)+\bar ss \varphi(t))\varphi(t), s\varphi(t)\rangle\ \ud \bar s\ \ud s\\
    & = \int_{0}^{1}\int_{0}^{1} \frac{\langle \nabla^2 U(z(t))\varphi(t), s\varphi(t)\rangle}{\|r_0(t) + \bar ss\varphi(t)\|_\mathcal{M}^{2+\alpha}}\ \ud \bar s\ \ud s\\
    & \le C'\int_{0}^{1} \int_{0}^{1} \frac{\|\varphi(t)\|_\mathcal{M}^2}{\|r_0(t) + \bar ss\varphi(t)\|_\mathcal{M}^{2+\alpha}}\ \ud\bar s\ \ud s\\
    & \le C'\|\varphi\|_\mathcal{D}^2 \int_{0}^{1} \int_{0}^{1} \frac{t}{\|r_0(t) + \bar ss\varphi(t)\|_\mathcal{M}^{2+\alpha}}\ \ud\bar s\ \ud s\\
    &\lesssim\frac{C''}{t^{1+\alpha}}\in L^1(1,+\infty),
\end{split}
\]
for proper $C',C''\in\R$ and $z(t) = \frac{r_0(t)+\bar ss \varphi(t)}{\|r_0(t)+\bar ss \varphi(t)\|_\mathcal{M}}\rightarrow \hat a=\frac{a}{\|a\|_\mathcal{M}}\in\Omega$, as $t\rightarrow+\infty$. Indeed, by the homogeneity of the potential, it holds
\[
\nabla^2 U (r_0(t) + \bar s s\varphi(t)) = \frac{\nabla^2 U (\frac{r_0(t) + \bar s s\varphi(t)}{\|r_0(t) + \bar s s\varphi(t)\|_\mathcal{M}})}{\|r_0(t) + \bar s s\varphi(t)\|_\mathcal{M}^{2+\alpha}} = \frac{\nabla^2 U(z(t))}{\|a\|_\mathcal{M}^{2+\alpha}t^{2+\alpha}(1 + o(t^{-\alpha}))},\quad t\rightarrow+\infty.
\]

Hence, there is a constant $C\in\R$ such that
\[
\left| U(r_0(t)+\varphi(t))-U(r_0(t))-\langle \ddot r_0(t), \varphi(t)\rangle_\mathcal{M}\right| \le \frac{C}{t^{\frac12+(m+1)\alpha}} \in L^1(1,+\infty),
\]
that proves the statement.

\noindent $\bullet$ The case $\alpha\in (0,2)$ and $r_0$ parabolic as in \eqref{eq:parabolic_eq}, the statement follows by observing, exactly as in the case $\alpha=1$, that there exist $C>0$ and $T \ge 1$ such that 
\begin{equation}
\label{eq:A}
\left|\langle
\nabla^2 U(r_0(t)+\psi(t))\psi(t),\psi(t)\rangle \right| =  \frac{\left|\langle\nabla^2 U(z(t))\psi(t), \psi(t)\rangle\right|}{\|r_0(t)+\psi(t)\|_\mathcal{M}^{2+\alpha}} \le C\, \frac{\|\psi(t)\|_\mathcal{M}^2}{t^2} \ \mbox{ if } \ t \ge T.
\end{equation}
Here, $\psi(t)=\varphi(t)+x-r_0(1)$, and $z(t)=\frac{r_0(t)+\psi(t)}{\|r_0(t)+\psi(t)\|_\mathcal{M}} \to b_m \notin \Delta$, as $t \rightarrow+\infty$. 

Thus, we have
\[
\begin{aligned}
U(r_0(t)+\psi(t))-U(r_0(t))-\langle \ddot r_0(t), \varphi(t)\rangle_{\mathcal M} &= \int_0^1 \langle \nabla U(r_0(t)+u\,\psi(t)), \psi(t)\rangle\, \ud u- \langle \nabla U(r_0(t)),\varphi(t)\rangle
\\
&= \int_0^1 \int_0^1\langle \nabla^2 U(r_0(t)+u\,v\,\psi(t))\psi(t), \psi(t)\rangle\,\ud u \,\ud v + \mbox{ l.o.t.}.
\end{aligned}
\]
Since $u, v \in [0,1]$, then from \eqref{eq:A} and Hardy's inequality, we get
\[
\int_{T}^{+\infty} \left| U(r_0(t)+\psi(t))-U(r_0(t))-\langle \ddot r_0(t), \varphi(t)\rangle_{\mathcal M}\right|\, \ud t \le C\int_T^\infty \frac{\|\psi(t)\|_\mathcal{M}^2}{t^2}\,\ud t \le C',
\]
since $\varphi \in \mathcal{D}_0^{1,2}(1,+\infty)$.

\noindent $\bullet$ The case $\alpha \in (1/2,2)$ and $r_0$ as in \eqref{eq:HP_eq}.


The well-definition of the renormalized action in the hyperbolic-parabolic setting follows from the same arguments as above, after decomposing the action as in \eqref{eq:decomposed_action}.

Indeed, the proof of the well-definition of the terms $\mathcal{A}_K$, for a cluster $K$, is the same as in the parabolic setting, while the well-definition of the terms $\mathcal{A}_{K_1,K_2}$, for pairs of clusters $K_1,K_2$, is proved as in the hyperbolic setting for $\alpha \in (1/2,2)$.
\end{proof}


In the next proposition we prove that $\mathcal A$ is differentiable (and continuous) on the the set of non-collision sets. 

\begin{prop}[Regularity of $\mathcal{A}$ over non-collision sets]
\label{prop:reg}
Assume either $\alpha >0$ and 
$r_0$ as in \eqref{eq:hyperbolic_r_0}, or $\alpha \in (0,2)$ and $r_0$ as in \eqref{eq:parabolic_eq}, or $\alpha \in (1/2,2)$ and $r_0$ as in \eqref{eq:HP_eq}. In all these settings, for any fixed initial configuration $x\in\mathcal{X}$, $\mathcal A$ is of class $C^1$ on the set
\[
\mathcal{D}_0^{1,2}(1,+\infty) \setminus \{\varphi\in\mathcal{D}_0^{1,2}(1,+\infty):\exists i<j,\ \exists t\in (1,+\infty)  \text{ such that } r_0(t)_{ij}+\varphi_{ij}(t)+x_{ij}-r_0(1)_{ij}=0\}.
\]
\end{prop}

\begin{proof}
Since the term $\frac{\|\varphi\|_\mathcal{D}^2}{2}$ is a smooth functional, we focus on the term
\[
\mathcal{A}^{pot}(\varphi) := \int_{1}^{+\infty} P(t,\varphi(t))\ \ud t = \int_{1}^{+\infty} U(r_0(t) + \varphi(t) + x - r_0(1)) - U(r_0(t)) - \langle \ddot r_0(t),\varphi(t)\rangle_\mathcal{M}\ \ud t.
\]

Differentiating the functional, it holds
\[
\ud \mathcal{A}^{pot}(\varphi)[\psi] = \int_{1}^{+\infty} \langle\nabla P(t,\varphi(t)),\psi(t)\rangle\ \ud t = \int_{1}^{+\infty} \langle\nabla U(r_0(t)+\varphi(t)+x-r_0(1))-\ddot r_0(t),\psi(t)\rangle\ \ud t, 
\]
for all $\psi\in\mathcal{D}_0^{1,2}(1,+\infty)$.

Consider a sequence $(\varphi_n)_n\subset\mathcal{D}_0^{1,2}(1,+\infty)$. We claim that if $\varphi_n\rightarrow\varphi$ in $\mathcal{D}_0^{1,2}(1,+\infty)$, then
\begin{equation}\label{eq:regularity_potential}
\lim_{n\rightarrow+\infty}\sup_{\|\psi\|_\mathcal{D}\leq 1} \bigg|\int_{1}^{+\infty} \langle \nabla P(t,\varphi_n(t))-\nabla P(t,\varphi(t)),\psi(t)\rangle\ \ud t\bigg| = 0.
\end{equation}

We prove \eqref{eq:regularity_potential} for all of three expansive settings.

\noindent $\bullet$ {\bf Case $\alpha>0$ and $r_0$ hyperbolic.}

In this case, by the definition of the guiding curve, it holds $\ddot r_0(t)=0$. Then
\[
\ud \mathcal{A}^{pot}(\varphi)[\psi] = \int_{1}^{+\infty} \langle\nabla P(t,\varphi(t)),\psi(t)\rangle\ \ud t = \int_{1}^{+\infty} \langle\nabla U(r_0(t)+\varphi(t)+x-r_0(1)),\psi(t)\rangle\ \ud t, 
\]
for all $\psi\in\mathcal{D}_0^{1,2}(1,+\infty)$.

It holds
\[
\begin{split}
\nabla P(t,\varphi_n(t))-\nabla P(t,\varphi(t)) & = \nabla U(r_0(t)+\varphi_n(t)+x-r_0(1)) - \nabla U(r_0(t)+\varphi(t)+x-r_0(1))\\
& = \int_{0}^{1} \langle \nabla^2 U (r_0(t) + \varphi(t) + x - r_0(1) + \sigma(\varphi_n(t) - \varphi(t)), \varphi_n(t) - \varphi(t)\rangle\ \ud \sigma\\
& \lesssim \frac{\|\varphi_n(t) - \varphi(t)\|_\mathcal{M}}{t^{2+\alpha}}.
\end{split}
\]

By Cauchy-Schwartz inequality and Proposition \ref{prop:compact_emb}, it follows
\[
\begin{split}
&\sup_{\|\psi\|_\mathcal{D}\leq 1} \bigg|\int_{1}^{+\infty} \langle \nabla U(r_0(t)+\varphi_n(t)+x-r_0(1)) - \nabla U(r_0(t)+\varphi(t)+x-r_0(1)),\psi(t)\rangle\ \ud t\bigg| \\
&\le \sup_{\|\psi\|_\mathcal{D}\leq 1} \int_{1}^{+\infty} t\| \nabla U(r_0(t)+\varphi_n(t)+x-r_0(1)) - \nabla U(r_0(t)+\varphi(t)+x-r_0(1)),\psi(t)\|_\mathcal{M}\frac{\|\psi(t)\|_\mathcal{M}}{t}\ \ud t\\
&\lesssim \sup_{\|\psi\|_\mathcal{D}\leq 1} \int_{1}^{+\infty} \frac{\|\varphi_n(t) - \varphi(t)\|_\mathcal{M}}{t^{1+\alpha}}\frac{\|\psi(t)\|_\mathcal{M}}{t}\ \ud t\\
&\le \sup_{\|\psi\|_\mathcal{D}\leq 1} \bigg(\int_{1}^{+\infty} \frac{\|\varphi_n(t) - \varphi(t)\|_\mathcal{M}^2}{t^{2(1+\alpha)}}\ \ud t\bigg)^{1/2}\bigg(\int_{1}^{+\infty}\frac{\|\psi(t)\|_\mathcal{M}^2}{t^2}\ \ud t\bigg)^{1/2}\\
& \le \frac{4}{(1+2\alpha)}\bigg(\int_{1}^{+\infty} \|\dot\varphi_n(t) - \dot\varphi(t)\|_\mathcal{M}^2\ \ud t\bigg)^{1/2}\\
& = \frac{4}{(1+2\alpha)}\|\varphi_n - \varphi\|_\mathcal{D}.
\end{split}
\]

\noindent $\bullet$ {\bf Case $\alpha \in (0,2)$ and $r_0$ parabolic}

Since
\[
    \nabla U(r_0(t)+\varphi(t)+x - r_0(1))-\nabla U(r_0(t)) = \int_{0}^{1} \langle\nabla^2 U(r_0(t)+s(\varphi(t)+x - r_0(1)),\varphi(t)+x-r_0(1)\rangle\ \ud s,
\]
we can estimate
\[
\begin{aligned}
    \|\nabla U(r_0(t)+\varphi(t)+x - r_0(1))-\nabla U(r_0(t))\|_\mathcal{M} &\lesssim \int_{0}^{1} \|\nabla^2 U(r_0(t)+s(\varphi(t)+x - r_0(1))\|_\mathcal{M} \|\varphi(t)\|_\mathcal{M}\ \ud s 
    \\
    &\leq C\frac{\|\varphi(t)\|_\mathcal{M}}{t^{\frac{2}{3}(2+\alpha)}},
    \end{aligned}
\]
for a proper constant $C$. 

By Cauchy-Schwartz inequality, it holds
\[
    \begin{split}
    &\sup_{\|\psi\|_\mathcal{D}\leq 1} \bigg|\int_{1}^{+\infty} \langle \nabla P(t,\varphi_n(t))-\nabla P(t,\varphi(t)),\psi(t)\rangle\ \ud t\bigg| \\
    &\leq \sup_{\|\psi\|_\mathcal{D}\leq 1} \int_{1}^{+\infty} t\| \nabla P(t,\varphi_n(t))-\nabla P(t,\varphi(t))\|_\mathcal{M}\frac{\|\psi(t)\|_\mathcal{M}}{t}\ \ud t\\
    &\leq \sup_{\|\psi\|_\mathcal{D}\leq 1} \bigg( \int_{1}^{+\infty} \frac{\|\psi(t)\|_\mathcal{M}^2}{t^2}\ \ud t\bigg)^{1/2} \bigg( \int_{1}^{+\infty} t^2 \| \nabla P(t,\varphi_n(t))-\nabla P(t,\varphi(t))\|_\mathcal{M}^2\ \ud t\bigg)^{1/2} \\
    & \leq 2 \bigg( \int_{1}^{+\infty} t^2 \| \nabla P(t,\varphi_n(t))-\nabla P(t,\varphi(t))\|_\mathcal{M}^2\ \ud t\bigg)^{1/2}.  
    \end{split}
\]

We have
\[
\begin{split}
    \| \nabla P(t,\varphi_n(t))-\nabla P(t,\varphi(t))\|_\mathcal{M}^2 &= \bigg| \int_{0}^{1} \nabla^2 P(t,\varphi(t)+\sigma(\varphi_n(t)-\varphi))(\varphi_n(t)-\varphi(t))\ \ud \sigma \bigg|^2\\
    & \leq \bigg( \int_{0}^{1} \|\nabla^2 P(t,\varphi(t)+\sigma(\varphi_n(t)-\varphi(t)))(\varphi_n(t)-\varphi(t))\|_\mathcal{M}\ \ud \sigma \bigg)^2\\
    & \leq \bigg( \int_{0}^{1} \frac{\|\varphi_n(t)-\varphi(t)\|_\mathcal{M}}{t^{\frac{2}{3}(2+\alpha)}}\ \ud \sigma \bigg)^2 \\
    & = \frac{\|\varphi_n(t)-\varphi(t)\|_\mathcal{M}^2}{t^{\frac{4}{3}(2+\alpha)}}.
\end{split}
\]

Then, using Proposition \ref{prop:compact_emb}, we have
\[
    \begin{split}
    \bigg( \int_{1}^{+\infty} t^2 \| \nabla P(t,\varphi_n(t))-\nabla P(t,\varphi(t))\|_\mathcal{M}^2\ \ud t\bigg)^{1/2} & \leq \bigg( \int_{1}^{+\infty} \frac{\|\varphi_n(t)-\varphi(t)\|_\mathcal{M}^2}{t^{\frac{2}{3}(1+2\alpha)}}\ \ud t\bigg)^{1/2}\\
    &\leq \bigg( \int_{1}^{+\infty} \frac{\|\varphi_n(t)-\varphi(t)\|_\mathcal{M}^2}{t^{(2+\alpha)}}\ \ud t\bigg)^{1/2}\\
    & \leq \frac{2}{(1+\alpha)} \bigg( \int_{1}^{+\infty} \|\dot{\varphi}_n(t)-\dot{\varphi}(t)\|_\mathcal{M}^2\ \ud t\bigg)^{1/2}\\
    & = \frac{2}{(1+\alpha)}\|\varphi_n - \varphi\|_\mathcal{D},
    \end{split}
\]
from which we conclude.

\smallskip
\noindent $\bullet$ {\bf Case $\alpha \in (1/2,2)$ and $r_0$ hyperbolic-parabolic.}

In the hyperbolic-parabolic setting, we exploit the decomposition of the action given by \eqref{eq:decomposed_action}. By studying the terms $\mathcal{A}_K$ and $\mathcal{A}_{K_1,K_2}$ independently, \eqref{eq:regularity_potential} can simply be proved by following the same arguments as the ones in the hyperbolic and parabolic settings. 
\end{proof}

As in \cite{PolimeniTerracini}, once it is proved that $\mathcal{A}$ is well-defined in $\mathcal{D}^{1,2}_0(1,+\infty)$, one aims to apply the {\em direct method}, as presented in Section~\ref{sec:existence}, in order to prove the existence of minimizers. Once this is done, we need to prove the following link between minimizers of $\mathcal{A}$ in $\mathcal{D}^{1,2}_0(1,+\infty)$ and free-time minimizers of the Lagrangian action. To this end, we state the following result, which parallels \cite[Corollary 6.3]{PolimeniTerracini}. 

Before giving the details, we observe that if the renormalization is not needed (that is, in the case $\alpha > 1$ and when $r_0$ is hyperbolic), the conclusion follows from Hamilton's Principle of Least Action combined with Marchal's Theorem on the absence of collisions for minimal trajectories.

\begin{prop}[Renormalized Action Principle]\label{th:ren_act_pr}
Fix an initial configuration $x \in \mathcal{X}$ and consider an expansive motion $\gamma(t)$ of the Newtonian $N$-body problem of the form 
    \begin{displaymath}
        \gamma(t)=r_0(t) + \varphi(t)+x-r_0(1), \ t \ge 1,
    \end{displaymath}
    where $\varphi\in\mathcal{D}_0^{1,2}(1,+\infty)$ minimizes the renormalized action \eqref{def:renormalized_action} in any of the settings of Theorems \ref{thm:hyp_alpha_expansive_1}, \ref{thm:hyp_alpha_expansive_2}, \ref{thm:par_alpha} and \ref{thm:HP_alpha} for $\alpha\in(0,1)$, or the Lagrangian action in any of the settings of Theorems \ref{thm:hyp_alpha_expansive_1}, \ref{thm:par_alpha} and \ref{thm:HP_alpha} {for the other admitted values of $\alpha$}.
    Then, $\gamma$ is a free-time minimizer of the Lagrangian action (and therefore it is a geodesic ray for the Jacobi-Maupertuis' metric) and solves Newton's equations \eqref{eq:NBP} for all $t \in (1,+\infty)$ (or for all $t \in [1, +\infty)$ if $x \in \Omega$).
\end{prop}

\begin{proof}
    Accordingly to the subclass of expansive motion under consideration, fix a reference path $r_0$.
    
    Suppose that a curve $\varphi\in\mathcal{D}_0^{1,2}(1,+\infty)$ is a minimizer of the renormalized Lagrangian action (or the Lagrangian action for the hyperbolic case with $\alpha>1$) and consider the associated expansive motion $\gamma(t) = r_0(t) + \varphi(t) + x - r_0(1)$.  By Hamilton's Principle of Least Action ({it follows from the differentiability of the renormalized action on non-collision sets, proved in Proposition \ref{prop:reg}}), since $\varphi$ is a minimizer, it is a solution of the Euler-Lagrange equations associated with $ \mathcal{A}$:
    \begin{equation}\label{eq:ELphi}
    \mathcal{M}\ddot{\varphi}(t) = \nabla U^\alpha(r_0(t) + \varphi(t) + x - r_0(1)) - \mathcal{M}\ddot{r}_0(t)
    \end{equation}
    at any time $t\in[1,+\infty)$ such that $\gamma(t)\in\Omega$. 

    \textbf{Claim:} $\gamma$ is a free-time minimizer for the Lagrangian action at its energy level.
    
    By the claim, we can apply Marchal's Principle (Theorem \ref{thm_marchal}), which states that $\gamma(t)$ has no collisions for any $t\in(1,+\infty)$.  {We can thus  conclude that $\varphi$ solves \eqref{eq:ELphi} and, equivalently, that $\gamma(t)$ is a solution of equations \eqref{eq:NBP} for any $t\in(1,+\infty)$.}

    For the hyperbolic case with $\alpha\in(0,1)$, and for the parabolic and hyperbolic cases, the proof of the claim can be found in \cite[Corollary 6.3]{PolimeniTerracini}. In the hyperbolic setting with $\alpha\in(1,+\infty)$, the claim can be proved using the same arguments as in \cite[Corollary 6.3]{PolimeniTerracini}, taking into account that in this case the renormalization of the action is not needed.

\end{proof}


\section{Existence of minimal expansive solutions}\label{sec:existence}

This section is devoted to proving the existence of minimizers of $\mathcal A$. These minimizers will yield the expansive solutions sought in this work. 

Our approach relies on the direct method of the calculus of variations: to apply this method, we establish that the action functional is coercive and weakly lower semicontinuous on the space $\mathcal{D}_0^{1,2}(1,+\infty)$. Both properties are analyzed separately for the hyperbolic, parabolic, and hyperbolic-parabolic regimes, taking into account the different possible values of the homogeneity parameter~$\alpha$.

\subsection{Coercivity estimates}\label{sec:coercivity}

In what follows, we will prove the coercivity estimates for $\mathcal{A}$ on $\mathcal{D}^{1,2}_0(1,+\infty)$.
 
 
\noindent $\bullet$ {\bf Case $\alpha \in (1,+\infty)$ and $r_0$ hyperbolic \eqref{eq:hyperbolic_r_0}}.

We prove the case $\alpha > 1$ first. We decompose the Lagrangian action, as follows
\begin{displaymath}
\begin{split}
    \mathcal{A}(\varphi) &= \int_{1}^{+\infty} \frac{1}{2}\|\dot\varphi(t)\|_\mathcal{M}^2+ U(at + \varphi(t) + x - a)\ \ud t\\
    & = \sum_{i<j} m_i m_j \int_{1}^{+\infty} \frac{|\dot{\varphi}_{ij}(t)|^2}{2M}+ \frac{1}{|a_{ij}t + \varphi_{ij}(t) + x_{ij} - a_{ij}|^\alpha}\ \ud t.
\end{split}
\end{displaymath}

For every pair $i\neq j$, direct manipulations, and the fact that $|\varphi(t)|\le \|\varphi\|_{\mathcal D} \sqrt{t}$, imply that 
\[
\begin{split}
    \int_{1}^{+\infty} \frac{1}{|a_{ij}t + \varphi_{ij}(t) + x_{ij} - a_{ij}|^\alpha}\ \ud t & \ge \int_{1}^{+\infty} \frac{1}{(|a_{ij}|t + |\varphi_{ij}(t)| + |x_{ij} - a_{ij}|)^\alpha}\ \ud t\\
    & \ge \int_{1}^{+\infty} \frac{1}{(|a_{ij}|t + \|\varphi_{ij}\|_\mathcal{D}t^{1/2} + |x_{ij} - a_{ij}|)^\alpha}\ \ud t\\
    & = \int_{1}^{+\infty} \frac{2s}{(|a_{ij}|s^2 + \|\varphi_{ij}\|_\mathcal{D}s + |x_{ij} - a_{ij}|)^\alpha}\ \ud s\\
    & =\frac{2}{|a_{ij}|^\alpha} \int_{1}^{+\infty} \frac{s}{\big(s^2 + \frac{\|\varphi_{ij}\|_\mathcal{D}}{|a_{ij}|}s + \frac{|x_{ij} - a_{ij}|}{|a_{ij}|}\big)^\alpha}\ \ud s\\
    & = \frac{2^{1+2\alpha}|a_{ij}|^\alpha}{\|\varphi_{ij}\|_{\mathcal{D}}^{2\alpha}} \int_{1}^{+\infty} \frac{s}{\bigg(\big( \frac{2|a_{ij}|}{\|\varphi_{ij}\|_{\mathcal{D}}}s +1 \big)^2 - 1 + \frac{4|x_{ij}-a_{ij}| |a_{ij}|}{\|\varphi_{ij}\|^2_{\mathcal{D}}}\bigg)^\alpha}\ \ud s\\
    & = \frac{2^{-1+2\alpha}|a_{ij}|^{-2+\alpha}}{\|\varphi_{ij}\|_{\mathcal{D}}^{2(-1+\alpha)}} \int_{\frac{2|a_{ij}|}{\|\varphi_{ij}\|_\mathcal{D}}}^{+\infty} \frac{\tau}{\bigg((\tau+1)^2-1+\frac{4|x_{ij}-a_{ij}| |a_{ij}|}{\|\varphi_{ij}\|^2_{\mathcal{D}}}\bigg)^\alpha}\ \ud \tau\\
    & \ge \frac{2^{-1+2\alpha}|a_{ij}|^{-2+\alpha}}{\|\varphi_{ij}\|_{\mathcal{D}}^{2(-1+\alpha)}} \int_{\frac{2|a_{ij}|}{\|\varphi_{ij}\|_\mathcal{D}}}^{+\infty} \frac{\tau}{(\tau+1)^{2\alpha}}\ \ud \tau,
\end{split}
\]
where the last inequality holds for $\|\varphi_{ij}\|_\mathcal{D}$ large enough. Now, we consider $\|\varphi_{ij}\|_\mathcal{D}$ large enough so that we can study the integral in the two intervals $\big[\frac{2|a_{ij}|}{\|\varphi_{ij}\|_\mathcal{D}},1\big]$ and $[1,+\infty)$. In $[1,+\infty)$, the integral is finite and does not depend on $\|\varphi_{ij}\|_\mathcal{D}$. Instead, if $\tau \in (0,1)$, then $\frac{\tau}{(\tau+1)^\alpha}\le \frac{\tau}{4^\alpha}$, hence
\[
\int_{\frac{2|a_{ij}|}{\|\varphi_{ij}\|_\mathcal{D}}}^{1} \frac{\tau}{(\tau+1)^{2\alpha}}\ \ud \tau \ge C\int_{\frac{2|a_{ij}|}{\|\varphi_{ij}\|_\mathcal{D}}}^{1} \tau\ \ud \tau = \frac{C}{2}\bigg(1-\frac{4|a_{ij}|^2}{\|\varphi_{ij}\|_\mathcal{D}^2}\bigg)\ge 0,
\]
if $\|\varphi\|_{\mathcal D}$ is large enough.

We have thus proved that
\begin{displaymath}
\begin{split}
    \mathcal{A}(\varphi) \ge  \sum_{i<j} m_i m_j \bigg( \frac{\|\varphi_{ij}\|_\mathcal{D}^2}{2M}+ \frac{C
    }{\|\varphi_{ij}\|_{\mathcal{D}}^{2(\alpha-1)}}\bigg),
\end{split}
\end{displaymath}
for a proper constant $C$.

\noindent $\bullet$ {\bf Case $\alpha \in (1/2,1)$ and $r_0$ hyperbolic \eqref{eq:hyperbolic_r_0}}.
In this case, we get
\[
\begin{split}
\mathcal{A}(\varphi) &= \int_{1}^{+\infty} \frac{1}{2}\|\dot\varphi(t)\|_\mathcal{M}^2+ U(at + \varphi(t) + x - a)-U(at)\ \ud t \\
&= \sum_{i<j} m_i m_j \int_{1}^{+\infty} \frac{|\dot{\varphi}_{ij}(t)|^2}{2M}+ \frac{1}{|a_{ij}t + \varphi_{ij}(t) + x_{ij} - a_{ij}|^\alpha}-\frac{1}{|a_{ij}t|^\alpha}\ \ud t.
\end{split}
\]

Following the same arguments as before, for every pair $i\neq j$, it holds that, provided that $\|\varphi_{ij}\|_\mathcal{D}$ large enough
\[
    \int_{1}^{+\infty} \frac{1}{|a_{ij}t + \varphi_{ij}(t) + x_{ij} - a_{ij}|^\alpha}-\frac{1}{|a_{ij}t|^\alpha}\ \ud t 
    \ge \frac{2^{-1+2\alpha}|a_{ij}|^{-2+\alpha}}{\|\varphi_{ij}\|_{\mathcal{D}}^{2(-1+\alpha)}} \int_{\frac{2|a_{ij}|}{\|\varphi_{ij}\|_\mathcal{D}}}^{+\infty} \bigg(\frac{1}{(\tau+1)^{2\alpha}}-\frac{1}{\tau^{2\alpha}}\bigg)\tau\ \ud \tau.
\]
Again, we study the integral of the two two intervals $[1,+\infty)$, where it is finite and does not depend on $\|\varphi_{ij}\|_\mathcal{D}$, and on $\big[\frac{2|a_{ij}|}{\|\varphi_{ij}\|_\mathcal{D}},1\big]$. In this last interval, it holds
\[
\int_{\frac{2|a_{ij}|}{\|\varphi_{ij}\|_\mathcal{D}}}^{1} \bigg(\frac{1}{(\tau+1)^{2\alpha}}-\frac{1}{\tau^{2\alpha}}\bigg)\tau\ \ud \tau \sim \int_{\frac{2|a_{ij}|}{\|\varphi_{ij}\|_\mathcal{D}}}^{1} -\tau^{1-2\alpha}\ \ud \tau = -\frac{1}{2(1-\alpha)}\bigg(1+\frac{2^{2(1-\alpha)}|a_{ij}|^{2(1-\alpha)}}{\|\varphi_{ij}\|_\mathcal{D}^{2(1-\alpha)}}\bigg).
\]

So, in this case,
\begin{displaymath}
\begin{split}
    \mathcal{A}(\varphi) \ge  \sum_{i<j} m_i m_j \bigg( \frac{\|\varphi_{ij}\|_\mathcal{D}^2}{2M}- C\,\|\varphi_{ij}\|_{\mathcal D}^{2(1-\alpha)}+ C\bigg),
\end{split}
\end{displaymath}
for a proper constant $C$.

\noindent $\bullet$ {\bf Case $\alpha \in (0,1/2]$ and $r_0$ hyperbolic \eqref{eq:hyperbolic_r_0}.}

In this case, we provide two proofs of coercivity estimates, with technical restriction on $\|a\|_{\mathcal M}$. Alternatively, we  give a condition for coercivity that does not depend on the configuration $a$, but depends on the starting time of the hyperbolic motion (or by considering $\mathcal{A}$ defined on some interval $[t_0,+\infty)$). As we prove in Section \ref{s:app}, these estimates are enough to prove the existence part of Theorem \ref{thm:hyp_alpha_expansive_2}, for every $a \in \Omega$.

\noindent {\em Estimates for $\|a\|_{\mathcal{M}}$ large.}\ 
We have
\[
\begin{split}
\mathcal{A}(\varphi) &= \int_{1}^{+\infty} \frac{\|\dot\varphi(t)\|_\mathcal{M}^2}{2} + U(r_0(t) + \varphi(t) + \tilde x) - U(r_0(t)) - \langle\ddot r_0(t),\varphi(t)\rangle_\mathcal{M}\ \ud t\\
&=\int_{1}^{+\infty} \frac{\|\dot\varphi(t)\|_\mathcal{M}^2}{2} + U(at+\eta(t) + \varphi(t) + \tilde x) - U(at+\eta(t)) - \langle\ddot \eta(t),\varphi(t)\rangle_\mathcal{M}\ \ud t
\\
& =\int_{1}^{+\infty} \frac{\|\dot\varphi(t)\|_\mathcal{M}^2}{2} + U(a t +\eta(t) + \varphi(t) + \tilde x) - U(at +\eta(t)) -\langle \nabla U(at +\eta(t)), \varphi(t)\rangle +\langle \varepsilon(t), \varphi(t)\rangle_{\mathcal M}\ \ud t,
\end{split}
\]
where we set $m=\lfloor 1/(2\alpha)\rfloor$, $\eta(t)=\sum_{k=1}^{m}\Gamma_k\,t^{1-k\alpha}$, and where the term $\varepsilon(t)$ satisfies 
\[
\varepsilon(t)=O(t^{-1-(m+1)\alpha}) \ \mbox{ as } \ t \rightarrow+\infty.
\]
From the asymptotic behavior of $\varepsilon(t)$, we get that, for some constant $C>0$,
\[
\int_{1}^{+\infty} \langle \varepsilon(t), \varphi(t) \rangle_{\mathcal M}\, dt \ge -C\int_{1}^{+\infty} t^{-(m+1)\alpha} \frac{\varphi(t)}{t}\, \ud t \ge -C \left(\int_{1}^{+\infty} \frac{\|\varphi(t)\|_\mathcal{M}^2}{t^2}\, \ud t\right)^{1/2} \ge -C \, \|\varphi\|_{\mathcal{D}},
\]
from the fact that $2(m+1)\alpha =2(\lfloor\frac{1}{2\alpha }\rfloor+1)\alpha>1$ and Hardy's inequality. 

Moreover, 
denoting $\psi=\varphi + \tilde x\in\mathcal{D}_0^{1,2}(1,+\infty)$, we also have
\[
\begin{split}
&U(r_0(t) + \varphi(t) + \tilde x) - U(r_0(t))  - \langle\nabla U(r_0(t)) ,\varphi(t)\rangle_\mathcal{M}\\
& \int_{0}^{1}\int_{0}^{1} \langle\nabla^2 U (r_0(t) + \bar s s\psi(t))\psi(t),s\psi(t)\rangle\ \ud \bar s\ \ud s.
\end{split}
\]
By the homogeneity of the potential, we have
\[
\nabla^2 U (r_0(t) + \bar s s\psi(t)) = \frac{\nabla^2 U (\frac{r_0(t) + \bar s s\psi(t)}{\|r_0(t) + \bar s s\psi(t)\|_\mathcal{M}})}{\|r_0(t) + \bar s s\psi(t)\|_\mathcal{M}^{2+\alpha}} = \frac{\nabla^2 U(z(t))}{\|a\|_\mathcal{M}^{2+\alpha}t^{2+\alpha}(1 + o(t^{-\alpha}))},
\]
where $\|z(t)\|_{\mathcal M}=1$ and $z(t)\rightarrow \hat a = \frac{a}{\|a\|_\mathcal{M}}\in\Omega$ as $t\rightarrow+\infty$, and it holds
\[
\frac{\langle\nabla^2 U(\hat a)\psi(t),\psi(t)\rangle}{\|a\|_\mathcal{M}^{2+\alpha}t^{2+\alpha}(1 + o(t^{-\alpha}))} \ge -C\frac{\|\psi(t)\|_\mathcal{M}^2}{\|a\|_\mathcal{M}^{2+\alpha}t^{2+\alpha}},
\]
where $C=\max_{z\in\mathcal
E}\nabla^2 U(z)$. Clearly, if $\|a\|_\mathcal{M}$ is large enough, then $\frac{C}{\|a\|_\mathcal{M}^{2+\alpha}}$ is sufficiently small. Then, using Hardy inequality:
\[
\begin{split}
&\int_{1}^{+\infty}\frac{1}{2}\|\dot\psi(t)\|_\mathcal{M}^2 + U(r_0(t) + \psi(t)) - U(r_0(t))  - \langle\nabla U(r_0(t)) ,\psi(t)\rangle_\mathcal{M}\ \ud t\\
& \ge \int_{1}^{+\infty}\frac{1}{2}\|\dot\psi(t)\|_\mathcal{M}^2 -C\frac{\|\psi(t)\|_\mathcal{M}^2}{\|a\|_\mathcal{M}^{2+\alpha}t^{2+\alpha}}\ \ud t\\
& \ge  \bigg( \frac{1}{2} -  \frac{4C}{(1+\alpha)^2\|a\|_\mathcal{M}^{2+\alpha}}\bigg)\|\psi\|_\mathcal{D}^2,
\end{split}
\]
where the term $\frac{1}{2} -  \frac{4C}{(1+\alpha)^2\|a\|_\mathcal{M}^{2+\alpha}}$ is positive if $\|a\|_\mathcal{M}$ is large enough.

\smallskip
\noindent {\em Estimates with $t_0>1$ large.} \ For $t_0\ge1$, we work with hyperbolic motions $\gamma:[t_0,+\infty)\rightarrow\Omega$,
\[
\gamma(t) = at +\sum_{k=1}^{m}\Gamma_k\,t^{1-k\alpha}+ \varphi(t) + x - a t_0 - \sum_{k=1}^{m}\Gamma_k t_0,\quad \varphi\in\mathcal{D}_0^{1,2}(t_0,+\infty),
\]
with $m=\lfloor 1/(2\alpha)\rfloor$, and where
\[
\mathcal{D}_0^{1,2}(t_0,+\infty) = \{\varphi\in H^1([t_0,+\infty))\ :\ \varphi(t_0)=0,\ \int_{t_0}^{+\infty}\|\dot\varphi(t)\|^2\ \ud t < +\infty\}
\]
and 
\[
\|\varphi\|_\mathcal{D} = \bigg(\int_{t_0}^{+\infty}\|\dot\varphi(t)\|_\mathcal{M}^2\ \ud t\bigg)^{1/2}.
\]

Following the same arguments as above, we get 
\[
\begin{split}
&\int_{t_0}^{+\infty}\frac{1}{2}\|\dot\psi(t)\|_\mathcal{M}^2 + U(r_0(t) + \psi(t)) - U(r_0(t))  - \langle\nabla U(r_0(t)) ,\psi(t)\rangle_\mathcal{M}\ \ud t\\
& \ge \int_{t_0}^{+\infty}\frac{1}{2}\|\dot\psi(t)\|_\mathcal{M}^2 -C\frac{\|\psi(t)\|_\mathcal{M}^2}{\|a\|_\mathcal{M}^{2+\alpha}t^{2+\alpha}}\ \ud t\\
& \ge \int_{t_0}^{+\infty}\frac{1}{2}\|\dot\psi(t)\|_\mathcal{M}^2 -C\frac{\|\psi(t)\|_\mathcal{M}^2}{\|a\|_\mathcal{M}^{2+\alpha}t_0^\alpha t^2}\ \ud t\\
& \ge \int_{1}^{+\infty}\frac{1}{2}\|\dot\psi(t)\|_\mathcal{M}^2 - \frac{4C}{t_0^\alpha\|a\|_\mathcal{M}^{2+\alpha}}\|\dot\psi(t)\|_\mathcal{M}^2\ \ud t\\
& = \bigg( \frac{1}{2} -  \frac{4C}{t_0^\alpha\|a\|_\mathcal{M}^{2+\alpha}}\bigg)\|\psi\|_\mathcal{D}^2,
\end{split}
\]
where the term $\frac{1}{2} -  \frac{4C}{t_0^\alpha\|a\|_\mathcal{M}^{2+\alpha}}$ is positive if $t_0$ is large enough.


\noindent  $\bullet$ {\bf Case $\alpha \in (0,2)$ and $r_0$ parabolic \eqref{eq:parabolic_eq}.}

In this case, we get
\begin{equation}
\label{e:A1}
\begin{split}
    \mathcal{A}(\varphi) &= \int_{1}^{+\infty} \frac{\|\dot\varphi(t)\|_{\mathcal M}^2}{2} + U(\beta b_m t^{\frac{2}{2+\alpha}}+\varphi(t)+\Tilde{x}) - U(\beta b_m t^{\frac{2}{2+\alpha}}) - \langle \nabla U(\beta b_m t^{\frac{2}{2+\alpha}}),\varphi(t)\rangle\ \ud t \\
    & \ge \int_{1}^{+\infty}\frac{\|\dot\varphi(t)\|_{\mathcal M}^2}{2}  + \frac{U_{min}}{\|r_0(t)+\varphi(t)+\Tilde{x}\|^\alpha} - \frac{U_{min}}{\|r_0(t)\|^\alpha} + \frac{\langle U_{min}\mathcal{M}r_0(t),\varphi(t)\rangle}{\|r_0(t)\|^{2+\alpha}}\ \ud t.
\end{split}
\end{equation}

We write
\begin{displaymath}
    \|r_0(t)+\varphi(t)+\Tilde{x}\|^2 = u + v,
\end{displaymath}
where
\begin{displaymath}
    \begin{split}
        & u = \|r_0(t)\|^2,\\
        & v = 2\langle \mathcal{M}r_0(t),\varphi(t)\rangle + 2\langle \mathcal{M}\varphi(t),\Tilde{x}\rangle + 2\langle \mathcal{M} r_0(t),\Tilde{x}\rangle + \|\varphi(t)\|^2 + \|\Tilde{x}\|^2.
    \end{split}
\end{displaymath}

Then,
\begin{equation*}
    \begin{aligned}
        \|r_0(t)+\varphi(t)+\Tilde{x}\|^{-\alpha} & = (u+v)^{-\frac{\alpha}{2}}
        \\
        &= u^{-\frac{\alpha}{2}} -\frac{\alpha}{2}u^{-\frac{2+\alpha}{2}}v + \frac{\alpha(2+\alpha)}{4}\int_{0}^{1}\int_{0}^{1}\langle(u+t_1 t_2 v)^{-\frac{4+\alpha}{2}}v,v\rangle t_1\ \ud t_1\ \ud t_2.
    \end{aligned}
\end{equation*}
Since the integral in the last expression is positive, it follows
\[
    \begin{split}
    \|r_0(t)+&\varphi(t)+\Tilde{x}\|^{-\alpha} 
    \\
    &\geq u^{-\frac{\alpha}{2}} -\frac{1}{2}u^{-\frac{2+\alpha}{2}}v \\
    &= \|r_0(t)\|^{-\alpha} -\frac{\alpha}{2\|r_0(t)\|^{2+\alpha}}(2\langle \mathcal{M}r_0(t),\varphi(t)\rangle + 2\langle \mathcal{M}\varphi(t),\Tilde{x}\rangle + 2\langle \mathcal{M} r_0(t),\Tilde{x}\rangle + \|\varphi(t)\|^2 + \|\Tilde{x}\|^2)\\
    & = \|r_0(t)\|^{-\alpha} -\alpha\bigg(\frac{\langle \mathcal{M}r_0(t),\varphi(t)\rangle}{\|r_0(t)\|^{2+\alpha}} - \frac{\langle \mathcal{M}\varphi(t),\Tilde{x}\rangle}{\|r_0(t)\|^{2+\alpha}} - \frac{\langle \mathcal{M} r_0(t),\Tilde{x}\rangle}{\|r_0(t)\|^{2+\alpha}} -\frac{1}{2}\frac{\|\varphi(t)\|^2}{\|r_0(t)\|^{2+\alpha}} - \frac{1}{2}\frac{\|\Tilde{x}\|^2}{\|r_0(t)\|^{2+\alpha}}\bigg).
    \end{split}    
\]

Putting the last chain of inequalities in \eqref{e:A1}, we get
\begin{equation}
\label{e:AA2}
\mathcal{A}(\varphi) \ge \int_{1}^{+\infty} \frac{\|\dot\varphi(t)\|_\mathcal{M}^2}{2} - \alpha U_{min}\bigg( \frac{\langle \mathcal{M}\varphi(t),\Tilde{x}\rangle}{\|r_0(t)\|^{2+\alpha}} + \frac{\langle \mathcal{M} r_0(t),\Tilde{x}\rangle}{\|r_0(t)\|^{2+\alpha}} +\frac{1}{2}\frac{\|\varphi(t)\|^2}{\|r_0(t)\|^{2+\alpha}} + \frac{1}{2}\frac{\|\Tilde{x}\|^2}{\|r_0(t)\|^{2+\alpha}}\bigg) \ \ud t.
\end{equation}

The two terms in the right-hand side of \eqref{e:AA2} that do not depend on $\varphi$ are bounded by a constant:
\[
\begin{split}
&-\alpha\int_{1}^{+\infty} U_{min}\bigg(  \frac{\langle \mathcal{M} r_0(t),\Tilde{x}\rangle}{\|r_0(t)\|^{2+\alpha}} + \frac{\|\Tilde{x}\|^2}{2\|r_0(t)\|^{2+\alpha}}\bigg)\ \ud t\\
&\ge -\alpha\int_{1}^{+\infty} U_{min}\bigg(\frac{\|\Tilde{x}\|}{\beta^{1+\alpha} t^{\frac{2(1+\alpha)}{2+\alpha}}} + \frac{\|\Tilde{x}\|^2}{2\beta^{1+\alpha} t^{\frac{2(1+\alpha)}{2+\alpha}}}\bigg)\ \ud t \\
&\ge C.
\end{split}
\]
Moreover, the third addendum in \eqref{e:AA2} has at most linear growth in $\|\varphi\|_{\mathcal{D}}$, since 
\[
\begin{split}
-\alpha U_{min}\int_{1}^{+\infty} \frac{\langle \mathcal{M}\varphi(t),\Tilde{x}\rangle}{\|r_0(t)\|^{2+\alpha}}\ \ud t &\ge - \alpha U_{min}\int_{1}^{+\infty} \frac{\|\varphi(t)\|\|\Tilde{x}\|}{\beta^{2+\alpha}t^2}\ \ud t \\
&\ge - \frac{\alpha U_{min}\|\tilde x\|}{\beta^{2+\alpha}} \bigg(\int_{1}^{+\infty} \frac{\|\varphi(t)\|^2}{t^2}\ \ud t \bigg)^{1/2}\bigg(\int_{1}^{+\infty} \frac{1}{t^2}\ \ud t\bigg)^{1/2} \\
& \ge - \frac{2\alpha U_{min}\|\tilde x\|}{\beta^{2+\alpha}} \|\varphi\|_\mathcal{D},
\end{split}
\]
thanks to Hardy's inequality.

Lastly, it remains to treat the following term
\[
\begin{split}
    -\frac{\alpha U_{min}}{2}\int_{1}^{+\infty} \frac{\|\varphi(t)\|^2}{\|r_0(t)\|^{2+\alpha}}\ \ud t = -\frac{\alpha U_{min}}{2}\int_{1}^{+\infty} \frac{\|\varphi(t)\|^2}{\beta^{2+\alpha}t^2}\ \ud t \ge -\frac{4\alpha}{(2+\alpha)^2}\|\varphi\|_\mathcal{D}^2,
\end{split}
\]
that we cannot prevent to be (possibly) quadratic in $\|\varphi\|_{\mathcal{D}}$. 

However, all in all, we obtain


\[
\mathcal{A}(\varphi) \ge \bigg(\frac{1}{2}-\frac{4\alpha}{(2+\alpha)^2}\bigg)\|\varphi\|_\mathcal{D}^2 -\frac{2\alpha U_{min}\|\tilde x\|}{\beta^{2+\alpha}}\|\varphi\|_\mathcal{D} + C,
\]
where $C$ is a proper constant and the coefficient $\frac{1}{2}-\frac{4\alpha}{(\alpha+2)^2}$ is positive for all $\alpha\in(0,2)$.

This concludes the proof in this case.

\noindent $\bullet$ {\bf Case $\alpha\in (1/2,2)$ and $r_0$ hyperbolic-parabolic as in \eqref{eq:HP_eq}.}

By following exactly the same arguments adopted in the {\em pure} hyperbolic case above, and using the decomposition of the action given in \eqref{eq:decomposed_action}, we can show that coercivity estimates hold for $\mathcal{A}_K(\varphi)$. We prove that they hold for $\mathcal{A}_{K_1,K_2}(\varphi)$ as well. For simplicity, we will write $b_{ij}=b_{ij}^{K_{1,2}}$ and $\beta=\beta^{K_{1,2}}$. To this end, it is convenient to separate case $\alpha \in (1,2)$ from case $\alpha \in (1/2,1)$.

Assume first $\alpha \in (1,2)$, for which $\mathcal{A}(\varphi)$ is simply given by the {\em non-renormalized} version of the action.  Using the changes of coordinates $t=s^2$ and $s=\|\varphi\|_\mathcal{D}^{1/\alpha}u$, we get
\[
\begin{split}
    \int_{1}^{+\infty} &\frac{1}{|a_{ij}t + \beta b_{ij}t^{\frac{2}{2+\alpha}} + \varphi_{ij}(t) + \tilde x_{ij}|^\alpha}\ \ud t \\
    &\ge \int_{1}^{+\infty} \frac{1}{(|a_{ij}|t + |\beta b_{ij}|t^{\frac{2}{2+\alpha}}+ |\varphi_{ij}(t)| + |\tilde x_{ij}|)^\alpha}\ \ud t\\
    & \ge \int_{1}^{+\infty} \frac{1}{(|a_{ij}|t+ |\beta b_{ij}|t^{\frac{2}{2+\alpha}} + \|\varphi_{ij}\|_\mathcal{D}t^{1/2} + |\tilde x_{ij}|)^\alpha}\ \ud t\\
    & = \int_{1}^{+\infty} \frac{2s}{(|a_{ij}|s^2 + |\beta b_{ij}|s^{\frac{4}{2+\alpha}} + \|\varphi_{ij}\|_\mathcal{D}s + |\tilde x_{ij}|)^\alpha}\ \ud s\\
    & =\frac{2}{|a_{ij}|^\alpha} \int_{1}^{+\infty} \frac{s}{\big(s^2 + \frac{|\beta b_{ij}|}{|a_{ij}|}s^{\frac{4}{2+\alpha}} + \frac{\|\varphi_{ij}\|_\mathcal{D}}{|a_{ij}|}s + \frac{|\tilde x_{ij}|}{|a_{ij}|}\big)^\alpha}\ \ud s\\
    & = \frac{2\|\varphi_{ij}\|_\mathcal{D}^{2/\alpha}}{|a_{ij}|^\alpha} \int_{\frac{1}{\|\varphi_{ij}\|_\mathcal{D}^{1/\alpha}}}^{+\infty}
    \frac{u}{\bigg(\|\varphi_{ij}\|_\mathcal{D}^{2/\alpha}u^2 + \frac{|\beta b_{ij}|}{|a_{ij}|}\|\varphi_{ij}\|_\mathcal{D}^{\frac{4}{2+\alpha}}u^{\frac{4}{2+\alpha}}+\frac{\|\varphi_{ij}\|_\mathcal{D}^{\frac{1+\alpha}{\alpha}}}{|a_{ij}|}u+\frac{|\tilde x_{ij}|}{|a_{ij}|}\bigg)^\alpha}\ \ud u\\
    & = \frac{2}{|a_{ij}|^\alpha} \int_{\frac{1}{\|\varphi_{ij}\|_\mathcal{D}^{1/\alpha}}}^{+\infty}
    \frac{u}{\bigg(u^2 + \frac{|\beta b_{ij}|}{|a_{ij}|\|\varphi_{ij}\|_\mathcal{D}^{\frac{\alpha}{2+\alpha}}}u^{\frac{4}{2+\alpha}}+\frac{1}{|a_{ij}|\|\varphi_{ij}\|_\mathcal{D}^{\frac{1-\alpha}{\alpha}}}u+\frac{|\tilde x_{ij}|}{|a_{ij}|\|\varphi_{ij}\|_\mathcal{D}^{2/\alpha}}\bigg)^\alpha}\ \ud u\\
    &\ge \frac{2}{|a_{ij}|^\alpha} \int_{\frac{1}{\|\varphi_{ij}\|_\mathcal{D}^{1/\alpha}}}^{+\infty}
    \frac{u}{\bigg(u^2 + u^{\frac{4}{2+\alpha}}+u+1\bigg)^\alpha}\ \ud u,
\end{split}
\]
where the last inequality holds for $\|\varphi_{ij}\|_\mathcal{D}$ large enough. Taking $\|\varphi_{ij}\|_\mathcal{D}$ large, we split the integral: on $[1,+\infty)$, it is constant and does not depend on $\|\varphi_{ij}\|_\mathcal{D}$; on $\bigg[\frac{1}{\|\varphi_{ij}\|_\mathcal{D}^{1/\alpha}},1\bigg]$, it holds
\[
\int_{\frac{1}{\|\varphi_{ij}\|_\mathcal{D}^{1/\alpha}}}^{1} \frac{u}{\bigg(u^2 + u^{\frac{4}{2+\alpha}}+u+1\bigg)^\alpha}\ \ud u \ge c\, \int_{\frac{1}{\|\varphi_{ij}\|_\mathcal{D}^{1/\alpha}}}^{1} u\ \ud u = \frac{c}{2}\bigg(1 - \frac{1}{\|\varphi_{ij}\|_\mathcal{D}^{2/\alpha}}\bigg),
\]
for some positive constant $c$.

So, we proved
\[
\mathcal{A}_{K_1,K_2}(\varphi)_{ij} \ge C\|\varphi_{ij}\|_{\mathcal{D}}^2 + \frac{1}{|a_{ij}|^\alpha}\bigg(1 - \frac{1}{\|\varphi_{ij}\|_\mathcal{D}^{2/\alpha}}\bigg) + C \ge C\|\varphi_{ij}\|_{\mathcal{D}}^2+ C,
\]
for a proper constant $C$. This ends the proof of the case $\alpha\in (1,2)$.

Assume now $\alpha \in (1/2,1)$. In this case, $\mathcal{A}(\varphi)$ is defined as the renormalized version of the action. Taking this into account, we get
\[
\begin{split}
    \int_{1}^{+\infty} &\frac{1}{|a_{ij}t + \beta b_{ij}t^{\frac{2}{2+\alpha}} + \varphi_{ij}(t) + \tilde x_{ij}|^\alpha} - \frac{1}{|a_{ij}t|^\alpha}\ \ud t \\
    &\ge \int_{1}^{+\infty} \frac{1}{(|a_{ij}|t + |\beta b_{ij}|t^{\frac{2}{2+\alpha}}+ |\varphi_{ij}(t)| + |\tilde x_{ij}|)^\alpha} - \frac{1}{|a_{ij}t|^\alpha}\ \ud t\\
    & \ge \int_{1}^{+\infty} \frac{1}{(|a_{ij}|t+ |\beta b_{ij}|t^{\frac{2}{2+\alpha}} + \|\varphi_{ij}\|_\mathcal{D}t^{1/2} + |\tilde x_{ij}|)^\alpha} - \frac{1}{|a_{ij}t|^\alpha}\ \ud t\\
    & = \int_{1}^{+\infty} \frac{2s}{(|a_{ij}|s^2 + \|\varphi_{ij}\|_\mathcal{D}s + |\tilde x_{ij}| - \frac{1}{|a_{ij}s^2|^\alpha})^\alpha}\ \ud s\\
    & =\frac{2}{|a_{ij}|^\alpha} \int_{1}^{+\infty} \bigg(\frac{1}{\big(s^2 + \frac{\|\varphi_{ij}\|_\mathcal{D}}{|a_{ij}|}s + \frac{|\tilde x_{ij}|}{|a_{ij}|}\big)^\alpha} - \frac{1}{s^{2\alpha}}\bigg)s\ \ud s\\
    & = \frac{2\|\varphi_{ij}\|_\mathcal{D}^{2/\alpha}}{|a_{ij}|^\alpha} \int_{\frac{1}{\|\varphi_{ij}\|_\mathcal{D}^{1/\alpha}}}^{+\infty}
    \bigg(\frac{1}{\big(\|\varphi_{ij}\|_\mathcal{D}^{2/\alpha}u^2 + \frac{|\beta b_{ij}|}{|a_{ij}|}\|\varphi_{ij}\|_\mathcal{D}^{\frac{4}{2+\alpha}}u^{\frac{4}{2+\alpha}}+\frac{\|\varphi_{ij}\|_\mathcal{D}^{\frac{1+\alpha}{\alpha}}}{|a_{ij}|}u+\frac{|\tilde x_{ij}|}{|a_{ij}|}\big)^\alpha}-\frac{1}{\|\varphi_{ij}\|_\mathcal{D}^{2/\alpha}u^{2\alpha}}\bigg)u\ \ud u\\
    & = \frac{2}{|a_{ij}|^\alpha} \int_{\frac{1}{\|\varphi_{ij}\|_\mathcal{D}^{1/\alpha}}}^{+\infty}
    \bigg(\frac{1}{\big(u^2 + \frac{|\beta b_{ij}|}{|a_{ij}|\|\varphi_{ij}\|_\mathcal{D}^{\frac{\alpha}{2+\alpha}}}u^{\frac{4}{2+\alpha}}+\frac{1}{|a_{ij}|\|\varphi_{ij}\|_\mathcal{D}^{\frac{1-\alpha}{\alpha}}}u+\frac{|\tilde x_{ij}|}{|a_{ij}|\|\varphi_{ij}\|_\mathcal{D}^{2/\alpha}}\big)^\alpha}-\frac{1}{u^{2\alpha}}\bigg)u\ \ud u\\
    &\ge \frac{2}{|a_{ij}|^\alpha} \int_{\frac{1}{\|\varphi_{ij}\|_\mathcal{D}^{1/\alpha}}}^{+\infty}
    \bigg(\frac{1}{\big(u^2 + u^{\frac{4}{2+\alpha}}+u+1\big)^\alpha}-\frac{1}{u^{2\alpha}}\bigg)u\ \ud u,
\end{split}
\]
where the last inequality holds for $\|\varphi_{ij}\|_\mathcal{D}$ large enough. Taking $\|\varphi_{ij}\|_\mathcal{D}$ large, we split the integral once again: on $[1,+\infty)$, it is constant and does not depend on $\|\varphi_{ij}\|_\mathcal{D}$; on $\bigg[\frac{1}{\|\varphi_{ij}\|_\mathcal{D}^{1/\alpha}},1\bigg]$, it holds
\[
\begin{split}
\frac{2}{|a_{ij}|^{\alpha}}\int_{\frac{1}{\|\varphi_{ij}\|_\mathcal{D}^{1/\alpha}}}^{1} \bigg(\frac{1}{\big(u^2 + u^{\frac{4}{2+\alpha}}+u+1\big)^\alpha}-\frac{1}{u^{2\alpha}}\bigg)\ \ud u & \ge \frac{2}{|a_{ij}|^\alpha} \int_{\frac{1}{\|\varphi_{ij}\|_\mathcal{D}^{1/\alpha}}}^{+\infty} -u^{1-2\alpha}\ \ud u\\
& = \frac{1}{|a_{ij}|^\alpha(1-\alpha)}\bigg(\frac{1}{\|\varphi_{ij}\|_\mathcal{D}^{\frac{2(1-\alpha)}{\alpha}}}-1\bigg).
\end{split}
\]

So, we proved
\[
\mathcal{A}_{K_1,K_2}(\varphi)_{ij} \ge \int_{1}^{+\infty}\left(\frac{|\dot\varphi_{ij}(t)|^2}{2M}\right)\,\ud t + \frac{2}{|a_{ij}|^\alpha}\left( \frac{1}{\|\varphi_{ij}\|_\mathcal{D}^{\frac{2(1-\alpha)}{\alpha}}}-1\right) + C \ge \frac{\|\varphi_{ij}\|_{\mathcal{D}}^2}{2M}+C,
\]
for a proper constant $C$.


\subsection{Weak-lower semicontinuity of the action}\label{sec:wlsc}

In what follows, we prove the weak-lower semicontinuity of the action. Once again, we examine all the cases separately, and we recall that the cases corresponding to $\alpha=1$ were treated in \cite{PolimeniTerracini}.

\noindent $\bullet$ {\bf Case $\alpha >1$ and $r_0$ hyperbolic.}

This case is trivial, since the action does not need a renormalization, and hence the weak-lower semicontinuity of $\mathcal A$ on $\mathcal{D}^{1,2}_0(1,+\infty)$ directly follows from Fatou's Lemma.

\noindent $\bullet$ {\bf Case $\alpha \in (1/2,1)$ and $r_0$ hyperbolic.} 

Consider a sequence $(\varphi^n)_n$ in $\mathcal{D}_0^{1,2}(1,+\infty)$ converging weakly to some limit $\varphi\in\mathcal{D}_0^{1,2}(1,+\infty)$. By the properties of weak convergence the sequence is bounded on $\mathcal{D}_0^{1,2}(1,+\infty)$ and admits a subsequence $(\varphi^{n_k})_k$ converging uniformly on compact subsets of $[1,+\infty)$ (and hence pointwise in $[1,+\infty)$). 

The equality
\[
    \frac{1}{|a_{ij}t+  \varphi^n_{ij}(t) +x_{ij} - a_{ij} |} - \frac{1}{|a_{ij}t|} = -\alpha\int_{0}^{1} \frac{(a_{ij}t + s(\varphi^n_{ij}(t) +x_{ij} - a_{ij}))(\varphi^n_{ij}(t) +x_{ij} - a_{ij})}{|a_{ij}t + s(\varphi^n_{ij}(t) +x_{ij} - a_{ij})|^{2+\alpha}}\ \ud s
\]
is true for all $t$ sufficiently large, since it must be, for all $s\in(0,1)$,
\[
|a_{ij}t + s(x^0_{ij} - a_{ij} + \varphi^n_{ij}(t))| \geq |a_{ij}|t - s(|x^0_{ij} - a_{ij}| + \|\varphi^n_{ij}\|_\mathcal{D}\sqrt{t}) >0.
\]
We suppose that this happens for $\bar T = \bar T(k)$, with $k$ being a constant such that $|\varphi^n_{ij}(t)| \leq k\sqrt{t}$ for all $n$ and $t\in[1,+\infty)$.

On $[1,\bar T]$, we can use Fatou's Lemma and the pointwise convergence of the sequence to state that
\[
\int_{1}^{\bar{T}} \frac{1}{|a_{ij}t + \varphi^n_{ij}(t) + x_{ij} - a_{ij}|^\alpha}\ \ud t \leq \liminf_{n\rightarrow+\infty} \int_{1}^{\bar{T}} \frac{1}{|a_{ij}t + \varphi^n_{ij}(t) + x_{ij} - a_{ij}|^\alpha}\ \ud t.
\]

On $[\bar T,+\infty)$, we can use the dominated convergence Theorem. Indeed, since the sequence $(\varphi^n)_n$ is bounded, it holds:
\[
\begin{aligned}
    \int_{\bar T}^{+\infty} \!\!\frac{1}{|a_{ij}t+  \varphi^n_{ij}(t) +x_{ij} - a_{ij} |} - \frac{1}{|a_{ij}t|}\ \ud t &= -\alpha\!\int_{\bar T}^{+\infty} \!\int_{0}^{1} \!\frac{(a_{ij}t + s(\varphi^n_{ij}(t) +x_{ij} - a_{ij}))(\varphi^n_{ij}(t) +x_{ij} - a_{ij})}{|a_{ij}t + s(\varphi^n_{ij}(t) +x_{ij} - a_{ij})|^{2+\alpha}}\ \ud s\ \ud t\\
    & \le \alpha\int_{\bar T}^{+\infty} \int_{0}^{1} \frac{|\varphi^n_{ij}(t) +x_{ij} - a_{ij}|}{|a_{ij}t + s(\varphi^n_{ij}(t) +x_{ij} - a_{ij})|^{1+\alpha}}\ \ud s\ \ud t\\
    & = \alpha\int_{\bar T}^{+\infty} \int_{0}^{1} \frac{|\varphi^n_{ij}(t) +x_{ij} - a_{ij}|}{|(a_{ij}t + s(\varphi^n_{ij}(t) +x_{ij} - a_{ij}))^2|^{\frac{1+\alpha}{2}}}\ \ud s\ \ud t\\
    & \le \alpha3^{\frac{\alpha}{2}}\int_{\bar T}^{+\infty} \int_{0}^{1} \frac{|\varphi^n_{ij}(t) +x_{ij} - a_{ij}|}{(|a_{ij}t|^2 - s^2|\varphi^n_{ij}(t) +x_{ij} - a_{ij}|^2)^{\frac{1+\alpha}{2}}}\ \ud s\ \ud t\\
    & \le \alpha3^{\frac{\alpha}{2}}\int_{\bar T}^{+\infty} \int_{0}^{1} \frac{|kt^{1/2} +x_{ij} - a_{ij}|}{(|a_{ij}t|^2 - |kt^{1/2} +x_{ij} - a_{ij}|^2)^{\frac{1+\alpha}{2}}}\ \ud s\ \ud t\\
    &\sim \alpha3^{\frac{\alpha}{2}}\int_{\bar T}^{+\infty} \frac{kt^{1/2}}{|a_{ij}t|^{1+\alpha}}\ \ud t \\
    & = \alpha3^{\frac{\alpha}{2}}\int_{\bar T}^{+\infty} \frac{k}{|a_{ij}|^{1+\alpha} t^{\frac{1}{2}+\alpha}}\ \ud t, 
\end{aligned}
\]
where (for $\bar T$ large enough) we used the inequality
\[
\frac{|b+c|^2}{|b|^2 - |c|^2} \geq \frac{1}{3}, \qquad \text{for each }b,c\in\mathbb{R}^d \text{ such that }|b|\geq2|c|.
\]

We have thus proved that 
\[
     \mathcal{A}(\varphi) \leq \liminf_{n\rightarrow+\infty} \int_{1}^{+\infty} \frac{1}{2} \|\dot{\varphi}^n(t)\|_\mathcal{M}^2 + U(at + \varphi^n(t) + x - a) - U(at)\ \ud t.
\]

\noindent $\bullet$ {\bf Case $\alpha \in (0,1/2]$ and $r_0$ hyperbolic.}

We recall that our assumption in this case is that
\[
r_0(t) = at + \sum_{k=1}^{m} \Gamma_k t^{1-k\alpha},
\]
where $m=\lfloor 1/(2\alpha)\rfloor$. 

Since there exists $\bar T\ge1$ such that, for all $t\ge\bar T$ and for all $s\in(0,1)$, it holds 
\[
|a_{ij}t + \sum_{k=1}^{m} \gamma_{k,ij} t^{1-k\alpha} + s(\varphi_{ij}(t) + \tilde x_{ij})|\ge |a_{ij}|t - \sum_{k=1}^{m} |\gamma_{k,ij}| t^{1-k\alpha}  - s(|\varphi_{ij}(t)| + |\tilde x_{ij}|) > 0,
\]
for all $t\in[\bar T,\infty)$ we can write
\[
\frac{1}{|r_0(t)_{ij}+  \varphi^n_{ij}(t) +x_{ij} - a_{ij}|^\alpha} - \frac{1}{|r_0(t)_{ij}|^\alpha} = \int_{0}^{1} \frac{\ud}{\ud s} \frac{1}{|r_0(t)_{ij} + s(\varphi_{ij}(t)+\tilde x_{ij})|^\alpha}\ \ud s,
\]
where $(\varphi^n)_n$ is a sequence in $\mathcal{D}_0^{1,2}(1,+\infty)$ converging weakly to a function $\varphi\in\mathcal{D}_0^{1,2}(1,+\infty)$.

The proof of the weak-lower semicontinuity of the action is the same as in the hyperbolic case $\alpha \in (1/2,1)$, above: on $[1,\bar T]$ it follows from Fatou's Lemma and the pointwise convergence of  $(\varphi^n)_n$ on compact sets, while on $[\bar T,+\infty)$ it follows from the dominated convergence Theorem.

\noindent $\bullet$ {\bf Case $\alpha \in (0,2)$ and $r_0$ parabolic. }

Consider a sequence $(\varphi^n)_n\subset\mathcal{D}_0^{1,2}(1,+\infty)$ converging weakly in $\mathcal{D}_0^{1,2}(1,+\infty)$ to a function $\varphi$. 

We can write the renormalized action as:
\[
\begin{split}
\mathcal{A}(\varphi) = \int_{1}^{+\infty} &\frac{1}{2}\|\dot{\varphi}(t)\|_\mathcal{M}^2 + \frac{1}{2}\langle \nabla^2 U(r_0(t))\varphi(t),\varphi(t)\rangle\\
    &+ U(r_0(t)+\varphi(t)+\Tilde{x}) - U(r_0(t)) - \langle \nabla U(r_0(t)),\varphi(t)\rangle - \frac{1}{2}\langle \nabla^2 U(r_0(t))\varphi(t),\varphi(t)\rangle\ \ud t.
\end{split}
\]

We claim that the map 
\[
\varphi(t) \mapsto \bigg( \frac{1}{2}\int_{1}^{+\infty} \|\dot{\varphi}(t)\|_\mathcal{M}^2 + \langle \nabla^2 U(r_0(t))\varphi(t),\varphi(t)\rangle\ \ud t \bigg)^{1/2}
\]
is an equivalent norm to $\|\cdot\|_\mathcal{D}$. 

Indeed, for a Newtonian potential of degree $-\alpha$, the hessian matrix is
\[
\nabla^2 U(r_0) = -\alpha\frac{U(b_m)}{\|r_0\|_\mathcal{M}^{2+\alpha}}\mathcal{M} + \alpha(2+\alpha)\frac{U(b_m)}{\|r_0\|_\mathcal{M}^{4+\alpha}}\mathcal{M}r_0\otimes\mathcal{M}r_0 - (1+2\alpha)\frac{\nabla_{b_m}U(b_m)\otimes \mathcal{M}r_0(t)}{\|r_0\|_\mathcal{M}^{3+\alpha}} + \frac{\nabla_{b_m}^2 U(b_m)}{\|r_0\|_\mathcal{M}^{2+\alpha}}.
\]
By the homogeneity of the potential, it holds 
\[
\langle \nabla^2 U(r_0(t))\varphi(t),\varphi(t)\rangle \geq -\frac{2\alpha}{(2+\alpha)^2}\frac{\|\varphi(t)\|_\mathcal{M}^2}{t^2},\quad\forall t\in[1,+\infty)
\]
and, by Hardy inequality,
\[
\frac{1}{2}\int_{1}^{+\infty} \|\dot{\varphi}(t)\|_\mathcal{M}^2 + \langle \nabla^2 U(r_0(t))\varphi(t),\varphi(t)\rangle\ \ud t \geq \bigg( 1-\frac{8\alpha}{(2+\alpha)^2} \bigg)\|\varphi\|^2_\mathcal{D} = \frac{(-2+\alpha)^2}{2(2+\alpha)^2}\|\varphi\|_\mathcal{D}^2.
\]
        
Since there exists a constant $C>0$ such that,
\[
\langle \nabla^2 U(r_0(t))\varphi(t),\varphi(t)\rangle \leq C\frac{\|\varphi(t)\|_\mathcal{M}}{t^2}
\]
we use Hardy inequality once again to state that
\[
\frac{1}{2}\int_{1}^{+\infty} \|\dot{\varphi}(t)\|_\mathcal{M}^2 + \langle \nabla^2 U(r_0(t))\varphi(t),\varphi(t)\rangle\ \ud t \leq (1+4C)\|\varphi\|_\mathcal{D}^2,
\]
which proves the claim.

From the equivalence between the two norms, the term $\frac{1}{2}\int_{1}^{+\infty} \|\dot{\varphi}(t)\|_\mathcal{M}^2 + \langle \nabla^2 U(r_0(t))\varphi(t),\varphi(t)\rangle\ \ud t$ is weak lower semicontinuous.

Concerning the remaining term, there exists $\bar T$ large enough such that, for all $t\ge \bar T$, it holds 
\[
    \begin{split}
     & U(r_0(t)+\varphi(t)+\Tilde{x}) - U(r_0(t)) - \langle \nabla U(r_0(t)),\varphi(t)\rangle - \frac{1}{2}\langle \nabla^2 U(r_0(t))\varphi(t),\varphi(t)\rangle\ \ud t \\
    & = \int_{0}^{1}\int_{0}^{1}\int_{0}^{1} \langle \nabla^3 U(r_0(t) + \tau_1\tau_2\tau_3 (\varphi^n(t)+\Tilde{x}))(\varphi^n(t)+\Tilde{x}),\varphi^n(t)+\Tilde{x},\varphi^n(t)+\Tilde{x}\rangle \tau_1 \tau_2^2\ \ud \tau_1\ \ud \tau_2\ \ud \tau_3.
     \end{split}
\]
Then, using the boundedness of the sequence $(\varphi^n)_n$, there exists $C,C'>0$ such that
\[
\begin{split}
&\int_{\bar T}^{+\infty} \int_{0}^{1}\int_{0}^{1}\int_{0}^{1} \langle \nabla^3 U(r_0(t) + \tau_1\tau_2\tau_3 (\varphi^n(t)+\Tilde{x}))(\varphi^n(t)+\Tilde{x}),\varphi^n(t)+\Tilde{x},\varphi^n(t)+\Tilde{x}\rangle \tau_1 \tau_2^2\ \ud \tau_1\ \ud \tau_2\ \ud \tau_3\ \ud t \\
& \le C\int_{\bar T}^{+\infty} \frac{\|\varphi^n(t) - \tilde x\|^3}{t^{\frac{2(3+\alpha)}{2+\alpha}}}\\
& \le C'\int_{\bar T}^{+\infty} \frac{1}{t^\frac{6+\alpha}{2(2+\alpha)}}.
\end{split}
\]
The weak lower semicontinuity of the term above then follows from the dominated convergence Theorem.

On the other hand, on the interval $[1,\bar T]$, the weak-lower semicontinuity is proved from Fatou's Lemma and from the uniform convergence of $(\varphi^n)_n$ on compact sets, which follows from Ascoli-Arzela's Theorem.

\noindent $\bullet$ {\bf Case $\alpha \in (1/2,2)$ and $r_0$ hyperbolic-parabolic. }

To prove the weak-lower semicontinuity of the renormalized Lagrangian action in the hyperbolic-parabolic setting, we again refer to the cluster decomposition given by \eqref{eq:decomposed_action}. By doing so, we can study the two terms $\mathcal{A}_{K}$ and $\mathcal{A}_{K_1,K_2}$ independently, following the arguments we made above in the hyperbolic and parabolic cases, separately.



\subsection{Application of the Renormalized Action Principle}
\label{s:app}

Thanks to Sections \ref{sec:coercivity} and \ref{sec:wlsc}, we can prove the existence part of Theorems~\ref{thm:hyp_alpha_expansive_1},  \ref{thm:hyp_alpha_expansive_2}, \ref{thm:par_alpha} and \ref{thm:HP_alpha}. Indeed, an application of the direct method in the Calculus of Variations yields the existence of at least a minimizer of the (renormalized) Lagrangian action on $\mathcal{D}_0^{1,2}(1,+\infty)$ in all three expansive settings. The Renormalized Action Principle thus ensures that such minimizers are solutions of the equations
\[
\mathcal{M}\ddot{\varphi}(t) = \nabla U (r_0(t) + \varphi(t) + x - r_0(1)) - \mathcal{M}\ddot{r}_0(t), \ t>1,
\]
for fixed homogeneity parameter $\alpha$, initial configuration $x\in\mathcal{X}$ and guiding curve $r_0(t)$ (for hyperbolic motions and $\alpha>1$, this follows from Hamilton's Principle of Least Action combined with Marchal's Principle). To be precise, the case $\alpha\in (0,1/2]$ deserves a comment. In this case, in Section \ref{sec:coercivity}, we proved the coercivity of $\mathcal{A}$ when assuming $\|a\|_{\mathcal M}$ large enough (or, somehow equivalently, when the action is defined on a time interval $[t_0,+\infty)$, with $t_0$ large enough). To prove the existence part in Theorem \ref{thm:hyp_alpha_expansive_2} for every $a\in \Omega$ and $x \in \mathcal X$, we can use standard scaling invariance properties, as follows. Consider $x \in \mathcal{X}$ and $a \in \Omega$. Let $\lambda>0$ be such that $\lambda \|a\|_{\mathcal{M}}$ is large enough so that the direct method can be applied to prove the existence of hyperbolic motions $\xi$, with asymptotic velocity $\lambda a$ and such that $\xi(1)=y=\lambda^{-2/\alpha} x$:
\[
\xi(t)=r_0^{\lambda a}(t)+\varphi(t)+y-r_0^{\lambda a}(1), \ t>1,
\]
where $r_0^{\lambda a}(t)=\lambda a t+ \sum_{k=1}^{\lfloor1/(2\alpha)\rfloor}\Gamma_k(\lambda a) t^{1-k\alpha}$ and $\varphi \in \mathcal{D}^{1,2}_0(1,+\infty)$.

Since $\frac{2}{\alpha}-2\frac{2+\alpha}{\alpha}=-\frac{2}{\alpha}(1+\alpha)$, by time-translation invariance and the space-time scaling invariance of Newton's equations, we have that
\[
\gamma(t):=\lambda^{2/\alpha}\xi\left(1+\frac{t-1}{\lambda^{\frac{2+\alpha}{\alpha}}}\right), \ t \ge 1,
\]
is an expansive motion, with asymptotic velocity $a$, such that $\gamma(1)=x$. Moreover, 
for every $k \in \mathbb N$ and $\lambda >0$,
\[
\Gamma_k(\lambda a)=\lambda^{1-k(2+\alpha)}\Gamma_k(a),
\]
as one can see by induction on $k$, from which one can obtain
\[
\lambda^{2/\alpha}\,\Gamma_k(\lambda a) \left(1+\frac{t-1}{\lambda^{\frac{2+\alpha}{\alpha}}}\right)^{1-k\alpha}=\Gamma_k(a) t^{1-k\alpha}+\psi(t)+\lambda^{2/\alpha}\Gamma_k(\lambda a)-\Gamma_k(a),
\]
with $\psi \in \mathcal{D}^{1,2}_0(1,+\infty)$. Also, it holds
\[
\lambda^{2/\alpha} \lambda a \left(1+\frac{t-1}{\lambda^{\frac{2+\alpha}{\alpha}}}\right)=a(t-1)+\lambda^{\frac{2+\alpha}{\alpha}}a.
\]
Hence, we get
\[
\begin{aligned}
\gamma(t)&=\lambda^{2/\alpha}r_0^{\lambda a}\left(1+\frac{t-1}{\lambda^{\frac{2+\alpha}{\alpha}}}\right) + \lambda^{2/\alpha} \varphi\left(1+\frac{t-1}{\lambda^{\frac{2+\alpha}{\alpha}}}\right)+x-\lambda^{2/\alpha}r_0^{\lambda a}(1)
\\
&=r_0^a(t)+\tilde\varphi(t)+x-r_0^a(1),
\end{aligned}
\]
for some $\tilde \varphi \in \mathcal{D}^{1,2}_0(1,+\infty)$.

\begin{rem} 
\label{rem:dot_gamma_a}
    {
    \rm
    Since the argument is quite standard in this context, we provide the main steps to prove that, if, in all the cases of $\alpha$ and $r_0$ considered, there exists $\varphi \in \mathcal{D}^{1,2}_0(1,+\infty)$ minimizer of $\mathcal A$, and $\gamma(t)=r_0(t)+\varphi(t)+x-r_0(1)$, then 
    \[
    \lim_{t\rightarrow+\infty}\dot \gamma(t) = a.
    \]

    Indeed, by Proposition \ref{th:ren_act_pr} follows, in particular, that $\gamma$ satisfies for every $t>1$, the Newton's equations. In particular, for every $t>s$, and $i=1,\dots,N$, we have
    \[
    \dot \gamma_i(t)-\dot \gamma_i(s) =\int_{s}^{t} \ddot \gamma_i(\tau)\,d\tau = \sum_{j\neq i} \int_{s}^{t} \frac{\gamma_{ij}(\tau)}{|\gamma_{ij}(\tau)|^{2+\alpha}}.
    \]
    We claim that $\frac{\gamma_{ij}(\tau)}{|\gamma_{ij}(\tau)|^{2+\alpha}} \in L^1(T,+\infty)$, for $T$ large enough.

    To prove the claim, we consider the different cases. 

    If $\gamma(t)=at + o(t)$, with $a \in \Omega$, then  $\frac{1}{|\gamma_{ij}(\tau)|^{1+\alpha}} = O\left(\frac{1}{t^{1+\alpha}}\right) $; if $\gamma(t)=at+\beta b_m t^\frac{2}{2+\alpha}+o(t^\frac{2}{2+\alpha})$, with $a \in \Delta$, then $\frac{1}{|\gamma_{ij}(\tau)|^{1+\alpha}} = O\left(\frac1{t^\frac{2+2\alpha}{2+\alpha}}\right)$.

    Hence, $|\dot \gamma(t)-\dot \gamma(s)|\to 0$, as $s \rightarrow+\infty$. This implies that there exists $\gamma_\infty$ such that 
    \[
    \lim_{t \rightarrow+\infty}\dot\gamma(t) = \gamma_\infty.
    \]
    Since $\varphi \in \mathcal{D}_0^{1,2}(1,+\infty)$, there exists $t_k \rightarrow+\infty$ such that $\lim_{k \rightarrow+\infty} \dot \varphi(t_k)=0$.
    
    Hence, from
    \[
    \lim_{k\rightarrow+\infty}\dot\gamma(t_k)=\lim_{k\rightarrow+\infty}\left(\dot r_0(t_k) + \dot\varphi(t_k)\right)=a,
    \]
    we obtain $\gamma_\infty=a$.
    
    }
\end{rem}


\section{Asymptotic estimates on the perturbations from the reference paths}\label{sec:asymptotic_estimates}

The goal of this section is to give the sharpest estimate possible to the growth of the term $\gamma(t)-r_0(t)$ at infinity, where $\gamma(t)$ is the expansive motion under consideration and $r_0(t)$ the corresponding guiding curve.

We study the three different types of expansive motions independently and examine the possible values of $\alpha$.

\subsection{Hyperbolic motions}

Set $\alpha >0$ and $r_0$ as in \eqref{eq:hyperbolic_r_0}.

We want to prove the following.
\begin{prop}
\label{prop:asymp_H}
Assume $\alpha \in (0,+\infty)$, $\alpha\neq 1$, $x \in \mathcal X$. Let $\gamma$ be a hyperbolic solution of the homogeneous $N$-body problem with the form $\gamma(t) = r_0(t) + \varphi(t) + x - r_0(1)$, with $\varphi\in\mathcal{D}_0^{1,2}(1,+\infty)$ being a minimizer of $\mathcal A$.

    Then, it holds, for proper constant vectors $Q,Q'\in\mathcal{X}$, as $t \rightarrow+\infty$:
    \begin{align*}
    \label{eq:asym_hyp}
    \gamma(t)-r_0(t) &= -\frac{\mathcal M^{-1} \nabla U(a)}{\alpha(1-\alpha)}\,t^{1-\alpha} + Q + o(t^{1-\alpha}),\ &\text{ if } \alpha>1;
    \\
     \gamma(t)-r_0(t) &=-\frac{\mathcal M^{-1} \nabla U(a)}{\alpha(1-\alpha)}\,t^{1-\alpha}-\Gamma_2\,t^{1-2\alpha}+Q' + o(t^{1-2\alpha}),\ &\text{ if } \alpha \in (1/2,1);
     \\
     \gamma(t)-r_0(t) &=-\mathcal M^{-1} \nabla U(a)\,\log t+o(\log t),\ &\text{ if } \alpha=1/2;
     \\
     \gamma(t)-r_0(t)&=\Gamma_{P}\,t^{1-P\alpha}+o(t^{1-P\alpha}), \ &\text{ if } \ \alpha \in (0,1/2),
    \end{align*}
    with $P = \lfloor 1/(2\alpha)\rfloor + 1$.
    
\end{prop}

\begin{proof}
    The case $\alpha=1$ is treated in \cite{PolimeniTerracini}. We prove the remaining cases.

    \noindent $\bullet$ Assume first $\alpha >1/2$. As a consequence of Remark \ref{rem:dot_gamma_a}, we have that $\dot \varphi(t) \to 0$, as $t \rightarrow+\infty$. Hence,
    \begin{equation}
    \label{eq:BB}
    \lim_{t \rightarrow+\infty}\frac{\dot \varphi(t)}{1/t^{\alpha}}=-\frac1\alpha\lim_{t \rightarrow+\infty}\frac{\ddot \varphi(t)}{1/t^{1+\alpha}}=-\frac{1}{\alpha}\lim_{t \rightarrow+\infty}\frac{\mathcal M^{-1}\nabla U\left(a+\frac{\varphi(t)+\tilde x}{t}\right)}{t^{1+\alpha}}t^{1+\alpha}=-\frac{\mathcal M^{-1}\nabla U(a)}{\alpha},
    \end{equation}
    since $\varphi(t)/t \to 0$, and $a \in \Omega$. 
    
    From \eqref{eq:BB}, we have, for every $t>s$,
    \[
    \varphi(t)-\varphi(s)=-\frac1\alpha\int_s^t\left( \frac{\mathcal M^{-1}\nabla U(a)}{\tau^\alpha}+h(\tau)\right)\,\ud\tau,
    \]
    where $h(\tau)=o\left(\frac1{\tau^\alpha}\right)$, as $\tau \rightarrow+\infty$. Therefore, by integrating the right-hand side of the last identity, we obtain the asymptotics, as follows: if $\alpha >1$, we have
    \[
    \varphi(t)= -\frac{\mathcal M^{-1}\nabla U(a)}{\alpha(1-\alpha)}t^{1-\alpha}+\Gamma,
    \]
    for some constant vector $\Gamma$; if $\alpha \in (1/2,1)$, we get
    \begin{equation}
    \label{e:Q1}
    \varphi(t)= -\frac{\mathcal M^{-1}\nabla U(a)}{\alpha(1-\alpha)}t^{1-\alpha}+o(t^{1-\alpha}).
    \end{equation}
    With \eqref{e:Q1} in hand, in the case $\alpha \in (1/2,1)$, we can improve the asymptotics for $\varphi(t)$, by using that, actually,
    \[
    \ddot \varphi(t)t^{1+\alpha}=\mathcal{M}^{-1}\nabla U(a)+\frac{\mathcal{M}^{-1}\nabla^2 U(a)\Gamma_1}{t^{\alpha}}+o\left(t^{-\alpha}\right).
    \]
    The statement follows after
    integrating twice $\ddot \varphi(t)$.

    \noindent $\bullet$ Assume now $\alpha\in (0,1/2]$, and $\alpha \neq \frac{1}{2p}$, for every $p \in \mathbb{N}$, so that
    \[
    \ddot r_0(t)=\sum_{k=1}^{m}\Gamma_k\,t^{-(1+k\alpha)},
    \]
    where $m=\lfloor 1/(2\alpha)\rfloor$.
    
    Hence, from the equation for $\ddot \varphi(t)$, we get
    \[
    \begin{aligned}
        \mathcal M\ddot \varphi(t)&=\mathcal M \ddot \gamma(t)-\mathcal M\ddot r_0(t)
        \\
        &= \nabla U(r_0(t)+\varphi(t)+\tilde x)-\sum_{k=1}^{m}\mathcal M \Gamma_k\,t^{-(1+k\alpha)}
        \\
        &=\frac{1}{t^{1+\alpha}}\nabla U(a+\eta(t))-\sum_{k=1}^{m}\mathcal M \Gamma_k\,t^{-(1+k\alpha)},
    \end{aligned}
    \]
    where $\eta(t)=\sum_{k=1}^{m}\Gamma_k\,t^{-k\alpha}+\varphi(t)/t+\tilde x/t$. with $\varphi(t)/t = o(t^{-m \alpha})$, as $t \rightarrow+\infty$.

   {
   Since
    \[
    \begin{aligned}
    \nabla U(a+\eta(t))&=\nabla U(a)+\sum_{q=1}^{m}\frac{1}{q!}\nabla^{1+q}U(a)[\eta^1,\dots,\eta^q] + \frac{1}{m!}\nabla^{m+1}U (a)[\eta^1,\dots, \eta^{m+1}]+o(|\eta|^{q+2})
    \\
    &= \nabla U(a)+\sum_{q=1}^{m}\frac{1}{q!}\sum_{j_1+\cdots +j_q\le m}\nabla^{1+q}U(a)[\Gamma_{j_1},\dots,\Gamma_{j_q}]\,t^{-(j_1+\cdots +j_q)\alpha} 
    \\
    &\quad + \sum_{q=1}^{m}\frac{1}{q!}\sum_{p=1}^{q-1}\nabla^{1+q}U(a)[\Gamma_{j_1},\dots,\Gamma_{j_p},\varphi,\dots,\varphi]\,t^{-{\left[(j_1+\cdots +j_p)\alpha+\frac{q-p}{2}\right]}} + \nabla^2 U(a)\frac{\varphi}{t}+o(t^{-m\alpha})
    \\
    &= \nabla U(a)+\sum_{q=1}^{m}\frac{1}{q!}\sum_{j_1+\cdots+ j_q\le m}\nabla^{1+q}U(a)[\Gamma_{j_1},\dots,\Gamma_{j_q}]\,t^{-(j_1+\cdots +j_q)\alpha} + o(t^{-m\alpha}),
    \end{aligned}
    \]
    }
 because of the fact that, from $m\alpha < \frac12$, 
 \[
 \frac{\varphi(t)}{t}=O(t^{-1/2})=o(t^{-m\alpha})
 \]
 and
 \[
 t^{-{\left[(j_1+\cdots j_p)\alpha-\frac{q-p}{2}\right]}} = o(t^{-m\alpha}),
 \]
 since $(j_1+\cdots j_p)\alpha+\frac{q-p}{2} \ge \alpha+\frac12 >m\alpha$.

 Notice that, by the very definition of the sequence $\Gamma_k$, if we put together terms multiplying the same negative power of $t$, we obtain
 \[
 \nabla U(a+\eta(t))= \sum_{k=0}^{m-1}\mathcal R_k\,t^{-k\alpha} + \mathcal R_m\,t^{-m\alpha} + o(t^{-m\alpha}),
 \]
 where $\mathcal R_k$, defined in \eqref{eq:R}, is such that 
 \[
 \Gamma_{k}=-\frac{\mathcal M^{-1}\mathcal R_{k-1}}{k\alpha(1-k\alpha)} \ \mbox{ for every } k=1,\dots,m.
 \]
 Hence,
 \[
 \mathcal M\ddot \varphi(t)= \frac{1}{t^{1+\alpha}}\left[ \mathcal R_m\,t^{-m\alpha} + o(t^{-m\alpha})\right],
 \]
 as $t \rightarrow+\infty$. Hence, we proved that
 \begin{equation}
     \label{eq:hyper_varphi}
      \lim_{t\rightarrow+\infty}\frac{\ddot \varphi(t)}{1/t^z} = \mathcal M^{-1}\mathcal R_m,
 \end{equation}
 where $z=1+(m+1)\alpha \in (3/2,2)$, since $\alpha \in (0,1/2)$.
 Therefore, we get
 \[
\lim_{t\rightarrow+\infty} \frac{\dot \varphi(t)}{1/t^{z-1}}=-\frac1{z-1}\lim_{t\rightarrow+\infty}\frac{\ddot \varphi(t)}{1/t^{z}}=-\frac{\mathcal M^{-1}\mathcal R_m}{z-1}.
 \]
    Arguing as in the cases $\alpha >1/2$, we obtain
    \[
    \varphi(t) -\varphi(s)=-\frac1{z-1}\int_s^t \left(\frac{\mathcal M^{-1}\mathcal R_m}{\tau^{z-1}}+h(\tau)\right)\,\ud\tau
    \]
    with $h(\tau)=o\left(\frac1{\tau^{z-1}}\right)$, as $\tau \rightarrow+\infty$. This, in turn, implies
     \[
    \varphi(t)=-\frac{\mathcal M^{-1}\mathcal R_m}{(z-1)(2-z)}\,t^{1-(m+1)\alpha}+o(t^{1-(m+1)\alpha}),
    \]
    that is
    \[
    \varphi(t)=\Gamma_{m+1}\,t^{1-(m+1)\alpha}+o(t^{1-(m+1)\alpha}),
    \]
    with $1-(m+1)\alpha \in (0,1/2)$, and $1-(m+1)\alpha \to 1/2$, as $\alpha\to 0^+$.

    \noindent $\bullet$  Assume now $\alpha =1/(2p)$, for some $p \in \mathbb{N}$. Under this assumption, $m=p$.

In this case, the same computations performed in the case $1/\alpha \notin 2 \mathbb{N}$ lead to
\begin{equation}
\label{eq:CC}
\mathcal M\ddot \varphi(t)=\frac{1}{t^{1+\alpha}}\left[\mathcal R_p\,t^{-1/2}+ \nabla^2 U(a)\varphi(t)/t + o(t^{-1/2})\right].
\end{equation}
Since we do not know yet if $\varphi(t)=o(t^{1/2})$, we cannot directly apply the same computations as before. However, from the last identity for $\ddot \varphi(t)$, we can deduce that, for every $\varepsilon>0$ small,
\[
\lim_{t \rightarrow+\infty} \frac{\dot \varphi(t)}{1/t^{\frac12+\alpha-\varepsilon}}=\lim_{t \rightarrow+\infty}\frac{\ddot \varphi(t)}{1/t^{\frac32+\alpha-\varepsilon}}=0.
\]
This implies that 
\[
\frac{\|\varphi(t)-\varphi(s)\|_\mathcal{M}}{t^{1/2}}\le \frac{\int_s^t \|\dot \varphi(\tau)\|_\mathcal{M}\,\ud\tau}{t^{1/2}} \le C \frac{\int_s^t \tau^{-\frac12-\alpha+\varepsilon}\,\ud\tau}{t^{1/2}}    \le C t^{-\alpha+\varepsilon} \to 0 \ \mbox{ as } \ t \rightarrow+\infty,
\]
that in turn informs us that $\varphi(t)=o(t^{1/2})$. Thus, from \eqref{eq:CC}, we can proceed as in the cases $1/\alpha \notin 2\mathbb N$ to deduce that
\[
\lim_{t \rightarrow+\infty}\frac{\dot \varphi(t)}{1/t^{\frac12+\alpha}}=-\frac1{\frac12 +\alpha}\lim_{t \rightarrow+\infty}\frac{\ddot \varphi(t)}{1/t^{\frac32+\alpha}}=-\frac1{\frac12+\alpha}\mathcal M^{-1}\mathcal R_p,
\]
from which we get
\[
\varphi(t)-\varphi(s)=-\frac1{\frac12+\alpha}\int_s^t \left(\frac{\mathcal M^{-1}\mathcal R_p}{\tau^{\frac{1+2\alpha}{2}}}+h(\tau)\right)\,\ud\tau,
\]
with $h(\tau)=o(\tau^{-\frac{1+2\alpha}{2}})$. Finally, if $\alpha=1/2$, since $\mathcal R_p=\nabla U(a)$ ,we conclude 
\[
\varphi(t)=-\mathcal M^{-1} \mathcal \nabla U(a) \, \log(t)+o(\log(t)),
\]
as $t\rightarrow+\infty$; if $\alpha \in (0,1/2)$, we deduce that
\[
\varphi(t)=-\frac{\mathcal M^{-1}\mathcal R_p}{\frac14 - \alpha^2}\,t^{\frac{1-2\alpha}{2}}+o\left(t^{\frac{1-2\alpha}{2}}\right)=\Gamma_{m+1}\,t^{\frac{1-2\alpha}{2}}+o\left(t^{\frac{1-2\alpha}{2}}\right),
\]
as $t \rightarrow+\infty$. Note, since $m\alpha=1/2$, $\frac{1-2\alpha}{\alpha}=1-(m+1)\alpha$.

 All in all, we proved Proposition \ref{prop:asymp_H}.
\end{proof}

\subsection{Parabolic motions}
The goal of this section is to prove the following result.

\begin{prop}
\label{prop:asymptotic_par}
    For $\alpha \in (0,2)$, $x \in \mathcal X$, and a normalized minimal central configuration $b_m$, let $\gamma$ be a corresponding parabolic solution having the form $\gamma(t) = \beta b_mt^{\frac{2}{2+\alpha}}+\varphi(t)+x-\beta b_m$, with $\varphi\in\mathcal{D}_0^{1,2}(1,+\infty)$ minimizer of $\mathcal A$, $\beta = \big({\frac{(2+\alpha)^2}{2}U(b_m)}\big)^{\frac{1}{2+\alpha}}$, and $\gamma(1)=x$. 
    
    Then, it holds
    \begin{equation}
    \label{eq:asymp}
    \gamma(t)-\beta b_m\,t^\frac{2}{2+\alpha}\,  =O(t^\frac{\alpha}{2+\alpha}), \ \mbox{ for } \ t \rightarrow+\infty.
    \end{equation}
\end{prop}

Before providing the proof of Proposition \ref{prop:asymptotic_par}, we collect some preliminary tools.

We begin our arguments by considering the differential equation
\begin{equation}
    \label{eq:eq_with_sing}
    \ddot{y}(t) + \frac{\mu}{t^2}y(t) = 0,
\end{equation}
with a constant $\mu<1/4$. This equation has monomial solutions $t^\theta$ if
\[
\theta(\theta-1)+\mu=0,
\]
which is obtained for 
\[
\theta_\pm = \frac{1\pm\sqrt{1-4\mu}}{2}.
\]
Since $\mu<1/4$, there are then two distinct real solutions, with $\theta_- < 1/2 < \theta_+$. As a consequence, in the space $\mathcal{D}_0^{1,2}(1,+\infty)$, the solutions of \eqref{eq:eq_with_sing} have the form $c\,t^{\theta_-}$, for some $c\in \mathbb R$.

\smallskip

Consider now the differential equation
\begin{equation}
\label{eq:mu_f}
\ddot{y}(t) + \frac{\mu}{t^2}y(t) = f(t), \ t> T,
\end{equation}
with $\mu<1/4$ and a given function $f$. By the method of constants for differential equations, we obtain the general solution
\begin{equation}
    \label{eq:genereal_sol}
    y(t) = c_1 t^{\theta_+} + c_2 t^{\theta_-} + y_p(t),
\end{equation}
where
\begin{equation}
    \label{eq:particular_sol}
    y_p(t)= \frac{1}{\theta_+-\theta_-}\left(t^{\theta_+}\int_{T}^{t} \sigma^{\theta_-} f(\sigma)\ \ud \sigma - t^{\theta_-}\int_{T}^{t} \sigma^{\theta_+}f(\sigma)\ \ud \sigma\right)
\end{equation}

\begin{rem}
\label{rem:f_gamma}
    Suppose that, for some $q>0$, $f=O(t^{-q})$, for $t\rightarrow+\infty$. Then,
    \[
    \dot y_p(t)=\frac{1}{\theta_+-\theta_-}\left(\theta_+\,t^{\theta_+-1}\int^t\,\sigma^{\theta_-}\,f(\sigma)\,\ud\sigma - \theta_- \, t^{\theta_--1}\int^t \sigma^{\theta_+}\,f(\sigma)\,\ud\sigma\right) =O(t\,|f(t)|)=O(t^{1-q}).
    \]
    Hence, $y_p \in \mathcal{D}_0^{1,2}(T,+\infty)$ if $q>3/2$. 
    
    Moreover, we have  
    \[
    \|y_p(t)\|= O(t^{{\theta_+}+{\theta_-}+1-q})=O(t^{2-q}),
    \]
    for $t\rightarrow+\infty$.

    Since $t^{\theta_+} \notin \mathcal{D}^{1,2}(T,+\infty)$, if $q > 2-\theta_-$, we deduce that every solution $y$ of \eqref{eq:mu_f} lying in $\mathcal{D}^{1,2}(T,+\infty)$ must satisfy $y(t)=O(t^{\theta_-})$, as $t\to+\ \infty$. In the case $\theta_-=\alpha/(2+\alpha)$, then $q > 2-\theta_-$ reads as $q > (4+\alpha)/(2+\alpha)$.
\end{rem}

We recall the following proposition from the classical theory of partial differential equations, which will be needed in the next arguments.  
\begin{prop}[Maximum principle]\label{prop:max_principle}
Let $y\in\mathcal{D}^{1,2}(T,+\infty)$ be such that $
\ddot{y}(t) + \frac{\mu}{t^2}y(t) \ge f(t)$, for  $t>T$, with $\mu <1/4$ and $f \in L^2(T,+\infty;t^2\,\ud t)$.

Then,
\[
y(t) \le y(T)\bigg(\frac{t}{T}\bigg)^{\theta_-} + y_p(t), \ \mbox{ for } \ t\ge T.
\]
\end{prop}
\begin{proof}
    Define
    \[
    v(t) := \frac{y(T)}{T^{\theta_-}}\,t^{\theta_-}+y_p(t) - y(t).
    \]
    The function $v \in \mathcal{D}^{1,2}_0(T,+\infty)$  satisfies
    \[
    -\ddot{v}- \frac{\mu}{t^2}v \ge 0, \ t>T.
    \]
   Multiplying with the negative part of $v$, $v^-$, we get
    \[
    0\ge 
    \int_{T}^{+\infty} \left(\|\dot v^-\|^2 - \frac{\mu}{t^2}\|v^-\|^2\right) \ud t \ge \left(\frac14 -\mu\right)\int_T^{+\infty} \frac{\|v^-\|^2}{t^2}\,\ud t.
    \]
    This implies $v^-\equiv 0$. Then, $v\ge0$.
\end{proof}

\begin{rem}
\label{rem:corollary}
    A consequence of Proposition \ref{prop:max_principle} and Remark \ref{rem:f_gamma} is that, if $f=O(t^{-q})$, for some $q \ge 2-\theta_-$, and $y\in \mathcal{D}^{1,2}(T,+\infty)$ satisfies $\ddot y(t)+\frac{\mu}{t^2}y(t)\ge f(t)$, for $t >T$, then $y=O(t^{\theta_-})$, for $t\rightarrow+\infty$.
\end{rem}

\smallskip
We are now in the position to prove Proposition \ref{prop:asymptotic_par}. In the proof, we adopt the following notation.
Recall that, by homogeneity of the Newtonian potential, if $U(x)=\tilde U(x)\|x\|_{\mathcal M}^{-\alpha}$, with $\tilde U(x)=\tilde U(x/\|x\|_{\mathcal M})$, we have
\[
U(x) = \frac{\tilde U(x)}{\|x\|_{\mathcal M}^\alpha}=\frac{\tilde U(\hat x)}{\|x\|_{\mathcal M}^\alpha},
\]
and
\begin{equation}
\label{eq:hessian}{
\nabla^2 U(x) = - \alpha\frac{\tilde U(\hat x)}{\|x\|_{\mathcal M}^{2+\alpha}}\mathcal M+\alpha(\alpha+2)\frac{\tilde U(\hat x)}{\|x\|_{\mathcal M}^{2+\alpha}}\mathcal M\hat x\otimes \mathcal M\hat x - (1+2\alpha)\frac{\nabla\tilde U(\hat x)\otimes \mathcal M\hat x}{\|x\|_{\mathcal M}^{2+\alpha}}  +\frac{\nabla^2\tilde  U(\hat x)}{\|x\|_{\mathcal M}^{2+\alpha}},}
\end{equation}
where we denote $\hat x= x/\|x\|_{\mathcal M}$.

\begin{proof}[Proof of Proposition \ref{prop:asymptotic_par}]
    Denoting $r_0(t) = \beta b_m t^{\frac{2}{2+\alpha}}$ and 
    \[
    \psi(t):=\varphi(t)+x-r_0(1),
    \]
    we want to prove that $\|\psi(t)\|_\mathcal{M}=O(t^\frac{\alpha}{2+\alpha})$, as $t \rightarrow+\infty$. We structure the proof into several steps.

    \smallskip
    \emph{Step 1. Intermediate estimates.} \ 
    Define $\eta(t):=\|\psi(t)\|_\mathcal{M}=\left(\sum_i m_i \|\psi_i\|^2\right)^{1/2}$. We have:
\[
\begin{split}
\ddot \eta &= \frac{\langle\ddot \psi, \psi\rangle_\mathcal{M}}{\eta} + \frac{\|\dot \psi\|_\mathcal{M}^2}{\eta} - \frac{\left(\langle \dot \psi, \psi\rangle_\mathcal{M}\right)^2}{\eta^3}
\\
&\ge \frac{\langle \ddot \psi, \psi\rangle_\mathcal{M}}{\eta}
\\
&=\frac{\langle \nabla U(r_0 + \psi) - \nabla U(r_0),\,\psi\rangle}{\eta}\\
& = \frac{1}{\eta} \int_{0}^{1}\langle\nabla^2 U(r_0+s\psi)\psi,\,\psi\rangle \, \ud s.
\end{split}
\]
For the sake of readability, we define
\[
x_s(t):=r_0(t)+s\,\psi(t) \ \mbox{ and } \ \hat x_s(t)=\frac{x_s(t)}{\|x_s(t)\|_{\mathcal M}}=\frac{r_0(t)+s\,\psi(t)}{\|r_0(t)+s\,\psi(t)\|_{\mathcal M}}.
\]
From \eqref{eq:hessian}, we have
\[
\langle\nabla^2 U(x_s)\psi, \psi\rangle \ge  \frac{\langle\nabla^2\tilde  U(\hat x_s)\psi,\psi\rangle}{\|x_s\|_{\mathcal M}^{2+\alpha}} - \frac{\langle  \nabla\tilde U(\hat x_s), \psi\rangle \langle \hat x_s, \psi\rangle_{\mathcal M}}{\|x_s\|_{\mathcal M}^{2+\alpha}} - \alpha\frac{U(\hat x_s)}{\|x_s\|_{\mathcal M}^{2+\alpha}}\|\psi\|_{\mathcal M}^2.
\]
Since $\hat x_s \to b_m$, and $b_m$ is a local minimum for $\tilde U$, then, for every $\varepsilon>0$, there is $T$ such that for all $t\ge T$
\[
\langle\nabla^2 U(x_s(t))\psi, \psi\rangle \ge -(1-\varepsilon) \frac{\alpha\,U_{min}}{\|r_0\|_{\mathcal M}^{2+\alpha}} \|\psi\|_\mathcal{M}^2 = -(1-\varepsilon)\frac{2\alpha}{(2+\alpha)^2}\frac{\eta^2}{ {t^2}}.
\]

Putting this into the differential inequality for $\eta$, we get
\[
\ddot\eta \ge  -(1-\varepsilon)\frac{2\alpha}{(2+\alpha)^2}\frac{\eta}{t^2}, \ \mbox{ for } \ t \ge T,
\]
that is $\ddot \eta + \mu_\varepsilon \frac{\eta}{t^2} \ge 0$, for $t \ge T$, with
\[
\mu_\varepsilon = (1-\varepsilon)\frac{2\alpha}{(2+\alpha)^2}< \frac{2\alpha}{(2+\alpha)^2}<1/4.
\]
By means of Proposition \ref{prop:max_principle}, with $f=0$, we obtain, for every $\varepsilon>0$,
\begin{equation}
    \label{eq:gamma_eps}
\eta(t) \le C_T\,t^{\sigma(\varepsilon)}, \ t \ge T,
\end{equation}
where $C_T=\frac{\eta(T)}{T^{\sigma(\varepsilon)}}$ and
\[
\sigma(\varepsilon) = \frac{1-\sqrt{1-4\mu_\varepsilon}}{2} = \frac{\alpha+\bar\varepsilon}{2+\alpha},
\]
for $\bar\varepsilon >0$ that tends to $0$ if $\varepsilon \to 0^+$.

\smallskip
\emph{Step 2. Refined estimates on $\eta$.} \ We can go back to the equation for $\eta(t)=\|\psi(t)\|_{\mathcal M}$, as in \textit{Step 1}:
\[
\begin{aligned}
\ddot \eta &\ge \frac{ \int_0^1 \langle \nabla^2 U(x_s)\psi, \, \psi \rangle \, \ud s}{\|\psi\|_\mathcal{M}} \ge \int_0^1 \left(-\alpha\frac{\tilde U(\hat x_s)\|\psi\|_\mathcal{M}^2}{\|x_s\|_\mathcal{M}^{2+\alpha}\|\psi\|_\mathcal{M}} + \frac{\langle\nabla^2 \tilde U(\hat x_s)\psi, \psi\rangle}{\|x_s\|_\mathcal{M}^{2+\alpha}\|\psi\|_\mathcal{M}} + \frac{\langle \nabla\tilde U(\hat x_s), \psi\rangle \langle \hat x_s, \psi\rangle_\mathcal{M}}{\|x_s\|_{\mathcal M}^{2+\alpha}\|\psi\|_\mathcal{M}}\right) \ud s
\\
&=-\alpha\frac{\tilde U(b_m)}{\|r_0\|_\mathcal{M}^{2+\alpha}} \eta + \alpha\|\psi\|_\mathcal{M}\int_0^1 \left(\frac{\tilde U(b_m)}{\|r_0\|_\mathcal{M}^{2+\alpha}} -\frac{\tilde U(\hat x_s)}{\|x_s\|_\mathcal{M}^{2+\alpha}}\right)\,\ud s + \int_0^1 \frac{\langle\nabla^2 \tilde U(\hat x_s)\psi, \psi\rangle}{\|x_s\|_\mathcal{M}^{2+\alpha}\|\psi\|_\mathcal{M}}\,\ud s +\int_0^1 \frac{\langle \nabla\tilde U(\hat x_s), \psi\rangle \langle \hat x_s, \psi\rangle_\mathcal{M}}{\|x_s\|_\mathcal{M}^{2+\alpha}\|\psi\|_\mathcal{M}}\, \ud s
\\
&{\ge -\alpha\frac{U_{min}}{\|r_0\|_\mathcal{M}^{2+\alpha}} \eta + \alpha\|\psi\|_\mathcal{M}\int_0^1 \left(\frac{U_{min}}{\|r_0\|_\mathcal{M}^{2+\alpha}} -\frac{\tilde U(\hat x_s)}{\|r_0+s\,\psi_b\,b\|_\mathcal{M}^{2+\alpha}}\right)\,\ud s + I_2+I_3}
\\
&=-\frac{2\alpha}{(2+\alpha)^2} \frac{\eta}{t^2} + I_1+I_2+I_3,
\end{aligned}
\]
where $\psi_b:=\langle \psi(t),b_m\rangle_{\mathcal M}$.

We claim that $I_1+I_2+I_3=O(t^{-q})$, as $t \rightarrow+\infty$, , for some $q \ge \frac{4+\alpha}{2+\alpha}$.

About $I_1$,  we have
{
\[
\begin{aligned}
 \frac{U_{min}}{\|r_0\|_\mathcal{M}^{2+\alpha}}&-\frac{\tilde U(\hat x_s)}{\|r_0+s\psi_b\,b_m\|_\mathcal{M}^{2+\alpha}}
=
\frac{1}{\|r_0\|_\mathcal{M}^{2+\alpha}}
\left[
U_{min}
-
\frac{\tilde U(\hat x_s)}{
\left( 1 + s^2\frac{\psi_b^2}{\|r_0\|_\mathcal{M}^2} + 2s\psi_b \frac{\langle b_m,r_0\rangle_\mathcal{M}}{\|r_0\|_\mathcal{M}^2} \right)^{\frac{2+\alpha}{2}}}
\right]
\\[0.6em]
& \le \frac{1}{\|r_0\|_\mathcal{M}^{2+\alpha}}
\left[
U_{min}
-
\frac{\tilde U(\hat x_s)}{
\left( 1 + \frac{|\psi_b|^2}{\|r_0\|_\mathcal{M}^2} +  \frac{ 2|\psi_b|}{\|r_0\|_\mathcal{M}} \right)^{\frac{2+\alpha}{2}}}
\right]
\\[0.6em]
&=\frac{1}{\|r_0\|_\mathcal{M}^{2+\alpha}}\left[U_{min} - \tilde U(\hat x_s) \left(1 - \frac{(2+\alpha)}{2}\left(\frac{|\psi_b|^2}{\|r_0\|_\mathcal{M}^2} + \frac{2|\psi_b|}{\|r_0\|_\mathcal{M}}\right)\right)+O\left(\frac{\|\psi\|_\mathcal{M}}{\|r_0\|_\mathcal{M}}\right)\right]
\\
&=\frac{1}{\|r_0\|_\mathcal{M}^{2+\alpha}}\left[U_{min}-\left(U_{min}+s\frac{\langle \nabla^2\tilde U(b_m)(\psi-\psi_b\,b_m),\psi- \psi_b\,b_m\rangle}{2\|r_0\|_\mathcal{M}^2}\right)\left(1-(2+\alpha)\frac{|\psi_b|}{\|r_0\|_{\mathcal M}}\right)\right.
\\
&\quad \quad \quad +\left.O\left(\frac{\|\psi\|_\mathcal{M}}{\|r_0\|_\mathcal{M}}\right)\right]
\\
&=O\left(\frac{|\psi_b|}{\|r_0\|_{\mathcal M}^{3+\alpha}}\right),
\end{aligned}
\]
where we used, in particular, that
\begin{equation}\label{eq:x_hat}
\begin{aligned}
    \hat x_s &= \frac{r_0+s\,\psi}{\|r_0+s\,\psi\|_\mathcal{M}} = \frac{r_0}{\|r_0+s\,\psi\|_\mathcal{M}}+\frac{s\,\psi}{\|r_0+s\,\psi\|_\mathcal{M}}
    \\
    &=b_m+\left(\frac{b_m}{\|b_m+s\,\psi/\|r_0\|_\mathcal{M}\|_\mathcal{M}}-b_m\right) + \frac{s\,\psi}{\|r_0\|_\mathcal{M}} + O\left(\frac{\|\psi\|_\mathcal{M}^2}{\|r_0\|_\mathcal{M}^2}\right)
    \\
    &= b_m + \frac{s}{\|r_0\|_\mathcal{M}}\left(\psi-\psi_b\,b_m\right)+ O\left(\frac{\|\psi\|_\mathcal{M}^2}{\|r_0\|_\mathcal{M}^2}\right).
    \end{aligned}
\end{equation}
}
Therefore, by means of {\em Step 1}, we get
\[
\begin{aligned}
I_1&=\alpha\|\psi\|_\mathcal{M}\int_0^1\left(
\frac{U_{min}}{\|r_0\|_\mathcal{M}^{\alpha+2}}
-
\frac{\tilde U(\hat x_s)}{\|r_0+s\,\psi_b\,b_m\|_\mathcal{M}^{2+\alpha}}\right)\ud s
\\
& =
O\!\left(
\frac{|\psi_b|\|\psi\|_\mathcal{M}}{\|r_0\|_\mathcal{M}^{3+\alpha}}
\right) = O\!\left(
\frac{\|\psi\|_\mathcal{M}^2}{\|r_0\|_\mathcal{M}^{3+\alpha}}
\right) \le C\, t^{\frac{2\alpha}{2+\alpha}+\varepsilon}t^\frac{-2(3+\alpha)}{2+\alpha}= C\,t^{-\frac{6}{2+\alpha}+\varepsilon},
\end{aligned}
\]
for every $\varepsilon>0$.

About $I_2$, we make use of the following fact. For every $V \in C^4$, if $b$ is a local minimum of $V$, then there exists $c>0$ such that, for $\|y\|_\mathcal{M}$ small enough, we have
\[
\langle \nabla^2 V(b+y)\,y,\,y\rangle \ge -c \|y\|_\mathcal{M}^4.
\]

From \eqref{eq:x_hat}, denoting  $y(t)=s\frac{\psi-\psi_b\,b_m}{\|r_0\|_\mathcal{M}}$, we get
\[
\begin{aligned}
\frac{\langle\nabla^2 \tilde U(\hat x_s)\psi, \psi\rangle}{\|x_s\|_\mathcal{M}^{2+\alpha}\|\psi\|_\mathcal{M}} &= \frac{\langle \nabla^2 \tilde U(b_m+y)y,\,y\rangle}{\|x_s\|_\mathcal{M}^{\alpha}\|\psi\|_\mathcal{M}}+s\psi_b^2\frac{\langle \nabla^2 \tilde U(b_m+y) b_m, \, b_m\rangle}{\|x_s\|_\mathcal{M}^{2+\alpha}\|\psi\|_\mathcal{M}}
\\
&\ge -c \frac{\|y\|_\mathcal{M}^4}{\|r_0\|_\mathcal{M}^\alpha \|\psi\|_\mathcal{M}}-C\frac{|\psi_b|^2\,\|y\|_\mathcal{M}}{\|r_0\|_\mathcal{M}^{2+\alpha} \|\psi\|_\mathcal{M}}
\\
&= -c \frac{\|\psi\|_\mathcal{M}^3}{\|r_0\|_\mathcal{M}^{4+\alpha}}-C'\frac{|\psi_b|^2}{\|r_0\|_\mathcal{M}^{3+\alpha}}
\\
&\ge -c \frac{\|\psi\|_\mathcal{M}^3}{\|r_0\|_\mathcal{M}^{4+\alpha}}-C'\frac{\|\psi\|_\mathcal{M}^2}{\|r_0\|_\mathcal{M}^{3+\alpha}}
\\
&\ge -C''\,t^{-\frac{8-\alpha}{2+\alpha}+\varepsilon}-C''\,t^{-\frac{6}{2+\alpha}+\varepsilon}
\\
&\ge -C'''\,t^{{-\frac{8-\alpha}{2+\alpha}}+\varepsilon},
\end{aligned}
\]
for every $\varepsilon>0$. Here, we used {\em Step 1}, and the fact that, for $\|y\|_\mathcal{M}$ small enough,
\[
\langle \nabla^2 \tilde U(b_m+y)b_m,\,b_m\rangle = \int_0^1 \nabla^3 \tilde U(b_m+s\,y)[b_m,b_m,y]\,\ud s\ge -C\|y\|_\mathcal{M}\ge -C' \frac{\|\psi\|_\mathcal{M}}{\|r_0\|_\mathcal{M}}.
\]

About $I_3$, we observe that
\[
\begin{aligned}
    \langle \nabla \tilde U(b_m+y),\,\psi\rangle &= \int_0^1 \langle\nabla^2 \tilde U(b_m+s\,y)y,\,\psi-\psi_b\,b_m \rangle\,\ud s + \int_0^1\langle \nabla^2 \tilde U(b_m+s\,y)y,\,\psi_b\,b_m\rangle\,\ud s 
    \\
    & = \|r_0\|_\mathcal{M}\int_0^1 \langle\nabla^2 \tilde U(b_m+s\,y)y,\,y \rangle\,ds +\int_0^1\int_0^1 \nabla^3 \tilde U(b_m+s\,u\,y)[y,y,b_m]\,\ud s\,\ud u
    \\
    &\ge -c \|r_0\|_\mathcal{M}\|y\|_\mathcal{M}^4 -C \|y\|_\mathcal{M}^2\,|\psi_b|.
\end{aligned}
\]
Hence,
\[
\begin{aligned}
I_3 &\ge  -c \frac{\|r_0\|_\mathcal{M}\|y\|_\mathcal{M}^4}{\|r_0\|_\mathcal{M}^{2+\alpha}\|\psi\|} -C \frac{\|y\|_\mathcal{M}^2\,|\psi_b|}{\|r_0\|_\mathcal{M}^{2+\alpha}\,\|\psi\|_\mathcal{M}}
\\
&\ge -c \frac{\|\psi\|_\mathcal{M}^3}{\|r_0\|_\mathcal{M}^{5+\alpha}}-C'\frac{\|\psi\|_\mathcal{M}^2}{\|r_0\|_\mathcal{M}^{4+\alpha}}
\\
&\ge -C'' t^{-\frac{8-\alpha}{2+\alpha}+\varepsilon}.
\end{aligned}
\]
for every $\varepsilon>0$, by means again of {\em Step 1}.

\smallskip
\emph{Step 3. Conclusion.}
\ Putting together Steps 1--3, we have proved that $\eta$ satisfies
\[
\ddot \eta -\frac{2\alpha}{(2+\alpha)^2}\eta \ge f(t), 
\]
with $f(t)=O(t^{-\frac{8-\alpha}{2+\alpha}+\varepsilon})$, as $t\rightarrow+\infty$, for every $\varepsilon>0$. Hence, an application of Proposition \ref{prop:max_principle} and Remark \ref{rem:corollary} leads to  
\[
\eta(t) \le \frac{\eta(T)}{T^{\theta_-}}t^{\theta_-} + y_p(t)=O(t^{\theta_-})=O(t^\frac{\alpha}{2+\alpha}), \ \mbox{ as } \ t\rightarrow+\infty,
\]
that concludes the proof.

\end{proof}

Once obtained the asymptotic result for $\|\varphi(t)\|_\mathcal{M}$, we can give the following result for the component of $\varphi$ along $b_m$.

\begin{prop} Under the same assumptions of Proposition \ref{prop:asymptotic_par}, we have, for $t \rightarrow+\infty$,
    \begin{equation}
    \label{eq:estimate_psi_b}
|\langle\varphi(t), \, b_m\rangle_\mathcal{M}| =\begin{cases}
    O(t^{\max\{m_-, \, \frac{-2+2\alpha}{2+\alpha}\}})=o(1), \ &\mbox{ if } \ \alpha \in (0,1),
    \\
    O(t^{\frac{-2+2\alpha}{2+\alpha}}), \ &\mbox{ if } \ \alpha \in [1,2).
\end{cases}
\end{equation}
\end{prop} 
\begin{proof} 
We set again $\psi=\varphi +x-r_0(1)$, and $\psi_b(t):=\langle \psi(t), b_m\rangle_{\mathcal M}$. By differentiating, we obtain

\[
\begin{split}
    \ddot \psi_b(t) &= \sum_{i} m_i\ddot\psi(t)b_i\\
    & = \int_{0}^{1}\langle\nabla^2 U(r_0(t) + s\psi(t))\psi(t),b_m\rangle\ \ud s\\
    & = \langle\nabla^2U(r_0(t))\psi(t),b_m\rangle + \int_{0}^{1} \langle[\nabla^2 U(r_0(t) + s\psi(t)) - \nabla^2 U(r_0(t))]\psi(t),b_m\rangle\ \ud s\\
    & = \langle\nabla^2U(r_0(t))\psi(t),b_m\rangle + \int_{0}^{1} \int_{0}^{1}\langle\nabla^3 U(r_0(t) + su\psi(t))s\psi(t),\psi(t),b_m\rangle\ \ud s\ \ud u.
\end{split}
\]
Now, by using \eqref{eq:hessian}, with $\nabla \tilde U(b_m)=0$ and $\nabla^2 \tilde U(b_m) b_m=0$, we obtain
\[
\begin{aligned}
\langle\nabla^2U(r_0(t))\psi(t),b_m\rangle &= - \alpha\frac{\tilde U(b_m)}{\|r_0\|_\mathcal{M}^{2+\alpha}}\psi_b(t) + \alpha(2+\alpha)\frac{\tilde U(b_m)}{\|r_0\|_\mathcal{M}^{2+\alpha}}\psi_b(t)
\\
&=-\frac{2\alpha}{(2+\alpha)^2}\frac{\psi_b(t)}{t^2} + \frac{2\alpha(2+\alpha)}{(2+\alpha)^2}\frac{\psi_b(t)}{t^2} 
\\
&= \frac{2\alpha(1+\alpha)}{(2+\alpha)^2}\frac{\psi_b(t)}{t^2}.
\end{aligned}
\]
Hence, $\psi_b$ satisfies
\begin{equation}
    \label{eq:psi_b}
\ddot\psi_b(t) - \frac{A}{t^2}\psi_b(t) = f(t), \ t>1,
\end{equation}
with $A=\frac{2\alpha(1+\alpha)}{(2+\alpha)^2}$ and 
\[
f(t) = \int_{0}^{1}\langle\nabla^3 U(r_0(t) + su\psi(t))s\psi(t),\psi(t),b_m\rangle\ \ud s\ \ud u,
\]
where
\[
f(t)=O\left(\frac{\|\psi\|_\mathcal{M}^2}{\|r_0(t)\|_\mathcal{M}^{3+\alpha}}\right), \ \mbox{ for } \ t \rightarrow+\infty.
\]
Hence, by using \eqref{eq:asymp}, we have
\[
f(t)=O(t^\frac{2\alpha}{2+\alpha}\,t^{-\frac{2(3+\alpha)}{2+\alpha}})=O(t^{-\frac{6}{2+\alpha}}).
\]
By applying \eqref{eq:genereal_sol} and \eqref{eq:particular_sol}, we have
\[
\psi_b(t)=c_+ t^{m_+} +c_- t^{m_-} + y_p(t),
\]
where
\[
m_\pm = \frac{1\pm \sqrt{1+4\,A}}{2},
\]
that makes $m_+>1$ and $m_-<0$,
while
\[
y_p(t)=O(t^{2-\frac{6}{2+\alpha}})=O(t^\frac{-2+2\alpha}{2+\alpha}).
\]
Since $t^{m_+} \notin \mathcal{D}^{1,2}(1,+\infty)$, then necessarily $c_+=0$, and hence \eqref{eq:estimate_psi_b} follows.
\end{proof}

\subsection{Hyperbolic-parabolic motions}

By mixing the two previous cases, namely hyperbolic and parabolic, we are ready to prove the following result.

\begin{prop}
\label{prop:asymp_HP}
Assume $\alpha \in (1/2,2)$, $x \in \mathcal X$, $0\neq a \in \Delta$, $b_m$ minimizing central configuration, and $\beta=\sqrt[2+\alpha]{\frac{(2+\alpha)^2}{\alpha}U^\alpha(b_m)}$. Set $r_0(t)=at + \beta b_m t^\frac{2}{2+\alpha}$.

Let $\gamma$ be a hyperbolic-parabolic solution of the homogeneous $N$-body problem of the form $\gamma(t) = r_0(t) + \varphi(t) + x - r_0(1)$, $\gamma(1)=x$, and where $\varphi\in\mathcal{D}_0^{1,2}(1,+\infty)$ is a minimizer of $\mathcal A_x$.

Then, it holds, as $t \rightarrow+\infty$:
\begin{equation*}
    \gamma(t)-r_0(t)=O(t^{\delta}), 
\end{equation*}
with $\delta=\max\left\{1-\alpha,\frac{\alpha}{2+\alpha}\right\}$.
    
\end{prop}

\begin{proof}
    The case $\alpha=1$, is covered essentially in \cite{PolimeniTerracini}. The novelty is that we can apply, to the parabolic part of the $a$-cluster decomposition, the sharp estimates \eqref{eq:asymp}.

In the case $\alpha\in (1/2,2)\setminus\{1\}$ we recall that $r_0(t)=a\,t+b_m\,\beta\,t^{\frac{2}{2+\alpha}}$, where $a \in \Delta \setminus \{0\}$, and $b_m$ is a central configuration. Also, we are studying expansive solutions of the form
\[
\gamma(t)=r_0(t)+\psi(t),
\]
with $\psi \in \mathcal{D}^{1,2}(1,+\infty)$, $\psi(t)=\varphi(t)+x-r_0(1)$, where $\varphi$ is a minimizer of $\mathcal A_x$.

As in \cite{PolimeniTerracini}, for a cluster $K$, its center of mass $c_K=\frac1{M_K}\sum_{i\in K} m_i\,\gamma_i(t)$ has a hyperbolic expansion. Indeed,
\[
c_K(t)=a_K\,t +h(t),
\]
and
\[
\ddot c_K(t)=-\frac{\alpha}{t^{1+\alpha}}\sum_{i\in K}\sum_{j\notin K}\frac{a_{ij}}{|a_{ij}|^{2+\alpha}} =O(t^{-(2+\alpha)}).
\]
Hence,
\[
h(t)=O(t^{1-\alpha}).
\]

 Consider now an index $i\in K$, and let $y_i(t)$ be the motion of the $i$-th component with respect to the center of mass of the cluster, that is
\[
y_i(t)=\gamma_i(t)-c^K(t),
\]
that is
\[
y_i(t)=\beta_m\,b_i^K\,t^\frac{2}{2+\alpha}+\psi_i(t)-h(t)
\]
We have
\[
\begin{aligned}
    \ddot y_i(t)&=-\alpha\sum_{j\in K}m_j\frac{\gamma_i(t)-\gamma_j(t)}{|\gamma_i(t)-\gamma_j(t)|^{2+\alpha}}-\alpha\sum_{j\notin K}m_j\frac{\gamma_i(t)-\gamma_j(t)}{|\gamma_i(t)-\gamma_j(t)|^{2+\alpha}}-\ddot c_i^K(t)
    \\
    &=-\alpha\sum_{j\in K}m_j\frac{y_i(t)-y_j(t)}{|y_i(t)-y_j(t)|^{2+\alpha}}+O(t^{-(2+\alpha)}).
\end{aligned}
\]
Hence, we can apply the same computations as in the parabolic case, to deduce that
\[
\psi_i(t)-h(t)=O(t^\frac{\alpha}{2+\alpha}) \ \mbox{ that is}\ \psi_i(t)=O(t^\delta) \ \mbox{ with } \ \delta=\max\left\{1-\alpha,\frac{\alpha}{2+\alpha}\right\}.
\]
This concludes the proof.
\end{proof}


\bibliographystyle{siam}
\bibliography{main}

@article{BDFT,
  author =       "A. Boscaggin and W. Dambrosio and G. Feltrin and S. Terracini",
  title =        "Parabolic orbits in Celestial Mechanics: a functional-analytic approach", 
  journal =      "Proceedings of the London Mathematical Society",
  volume =       "123",
  number =       "2",
  pages =        "203-230",
  year =         "2021"
}

@article{Burgos_PartiallyHyperbolic,
  author =       "Burgos, J. M.",
  title =        "Existence of partially hyperbolic motions in the $N$-body problem", 
  journal =      "Proceedings of the American Mathematical Society",
  volume =       "150",
  number =       "4",
  pages =        "1729–1733",
  year =         "2022"
}

@article{Chazy,
  author =       "J. Chazy",
  title =        "Sur l'allure du mouvement dans le problème des trois corps quand le temps croit indéfiniment",
  journal =      "Annales scientifiques de l'École Normale Supérieure",
  volume =       "39",
  number =       "3",
  pages =        "29-130",
  year =         "1922"
}

@article{Chenciner,
  author =       "A. Chenciner",
  title =        "Action minimizing solutions of the {N}ewtonian $n$-body problem: from homology to symmetry", 
  journal =      "Proceedings of the International Congress of Mathematicians",
  volume =       "3",
  pages =        "279-294",
  year =         "2002"
}

@article{FerrarioTerracini,
  author =       "D. L. Ferrario and S. Terracini",
  title =        "On the existence of collisionless equivariant minimizers for the classical $n$-body problem", 
  journal =      "Inventiones mathematicae",
  volume =       "155",
  number =       "2",
  pages =        "305-362",
  year =         "2004"
}

@article {BarutelloFerrarioTerracini2008,
    AUTHOR = {Barutello, Vivina and Ferrario, Davide L. and Terracini,
              Susanna},
     TITLE = {On the singularities of generalized solutions to
              {$n$}-body-type problems},
   JOURNAL = {International Mathematics Research Notices},
  FJOURNAL = {International Mathematics Research Notices. IMRN},
      YEAR = {2008},
     PAGES = {1-78}
}

@article{MadernaVenturelli_GloballyMinimizingParabolic,
  author =       "E. Maderna and A. Venturelli",
  title =        "Globally minimizing parabolic motions in the {N}ewtonian $N$-body problem", 
  journal =      "Archive for Rational Mechanics and Analysis",
  volume =       "194",
  number =       "1",
  pages =        "283-313",
  year =         "2009"
}

@article{MadernaVenturelli_HyperbolicMotions,
  author =       "E. Maderna and A. Venturelli",
  title =        "Viscosity solutions and hyperbolic motions: a new PDE method for the $N$-body problem",
  journal =      "Annals of Mathematics",
  volume =       "192",
  pages =        "499-550",
  year =         "2020"
}

@article{Marchal_MethodOfMinimization,
  author =       "C. Marchal",
  title =        "How the method of minimization of action avoids singularities", 
  journal =      "Celestial Mechanics and Dynamical Astronomy",
  volume =       "83",
  number =       "1-4",
  pages =        "325-353",
  year =         "2002"
}

@article{MarchalSaari_FinalEvolution,
  author =       "C. Marchal and D. G. Saari",
  title =        "On the final evolution of the $n$-body problem", 
  journal =      "Journal of Differential Equations",
  volume =       "20",
  number =       "1",
  pages =        "150-186",
  year =         "1976"
}

@article {MR1610784,
    AUTHOR = {Montgomery, Richard},
     TITLE = {The {$N$}-body problem, the braid group, and action-minimizing
              periodic solutions},
   JOURNAL = {Nonlinearity},
  FJOURNAL = {Nonlinearity},
    VOLUME = {11},
      YEAR = {1998},
    NUMBER = {2},
     PAGES = {363--376},
      ISSN = {0951-7715,1361-6544},
   MRCLASS = {70F10 (20F36 34C25 58F22)},
  MRNUMBER = {1610784},
MRREVIEWER = {Christopher\ K.\ McCord},
       DOI = {10.1088/0951-7715/11/2/011},
       URL = {https://doi.org/10.1088/0951-7715/11/2/011},
}

@article{paradela2022oscillatory,
      title={Oscillatory Motions in the Restricted 3-body Problem: A functional analytic approach}, 
      author={Jaime Paradela and Susanna Terracini},
      year={2022},
      eprint={2212.05684},
      journal = "preprint on arXiv database 	arXiv:2212.05684",
      archivePrefix={arXiv},
      primaryClass={math.DS}
}

@article{PolimeniTerracini,
      title= "On the existence of minimal expansive solutions to the N-body problem", 
      author="Davide Polimeni and Susanna Terracini",
      journal = "Inventiones Mathematicae",
      volume = "238", 
      pages = "585–635",
      year= "2024"
}

@article{Poincare_SolutionsPeriodiques,
  author =       "H. Poincaré",
  title =        "Sur les solutions périodiques et le principe de moindre action", 
  journal =      "Comptes rendus hebdomadaires des séances de l'Académie des sciences de Paris",
  volume =       "123",
  pages =        "915-918",
  year =         "1896"
}

@article{Pollard_BehaviorOfGravitationalSystems,
  author =       "H. Pollard",
  title =        "The behavior of gravitational systems", 
  journal =      "Journal of Mathematics and Mechanics",
  volume =       "17",
  pages =        "601-611",
  year =         "1967"
}

@article{Poincare_SurLeProblemes,
  author =       "H. Poincaré",
  title =        "Sur le problème des trois corps et les équations de la dynamique", 
  journal =      "Acta Mathematica",
  volume =       "13",
  pages =        "1-270",
  year =         "1890"
}

@article{BertiPolimeniTerracini,
  author =       "D. Berti and D. Polimeni and S. Terracini",
  title =        "On the regularity of solutions to the Hamilton-Jacobi equations for the N-body problem",
  year =      "2025",
  journal = "preprint on arXiv database arXiv:2507.1917",
  URL = {https://arxiv.org/abs/2507.19170}
}

@article{Yu_hyperbolic,
  author =       "G. Yu",
  title =        "Hyperbolic motions in the {N}-body problem with homogeneous potentials",
  year =      "2024",
  journal = "Discrete and Continuous Dynamical Systems",
  volume = "44",
  number = "12",
  pages = "3698-3708",
  DOI = "10.3934/dcds.2024074"
}

@article{LiuYanZhou,
  author =       "J. Liu and D. Yan and Y. Zhou",
  title =        "Existence of hyperbolic motions to a class of Hamiltonians and generalized {N}-body system via a geometric approach",
  year =      "2023",
  journal = "Archive for Rational Mechanics and Analysis",
  volume = "247",
  number = "4"
}

\end{document}